\documentclass[12pt]{amsart}
\usepackage{amsfonts}
\usepackage{amsmath}
\usepackage{amssymb}
\usepackage[margin=1.2in]{geometry}
\usepackage{verbatim}
\usepackage{graphicx,lscape}
\usepackage{accents}
\newcommand{\ubar}[1]{\underaccent{\bar}{#1}}

\newtheorem{theorem}{Theorem}[section]

\newtheorem{lemma}{Lemma}[section]

\theoremstyle{remark}
\newtheorem{remark}[theorem]{Remark}

\newcommand {\D} {\displaystyle}
\usepackage{bm}

\newcommand{\half}{\frac{1}{2}}
\newcommand {\imbed} {\hookrightarrow}

\newcommand {\scal}[2]{\left(#1,#2\right)}

\newcommand{\abs}[1]{\left\lvert #1 \right\rvert}
\newcommand{\vnorma}[1]{\left\|#1\right\|}

\DeclareMathOperator{\di}{d\hspace{-1.5pt}}
\newcommand{\dx}{ \di x}
\newcommand{\dX}{ \di \X}
\newcommand {\dt}{ \di t}
\newcommand{\ds}{\di s}

\newcommand {\NN } {{\mathbb N}}
\newcommand {\RR } {{\mathbb R}}
\newcommand{\veps}{\varepsilon}

\newcommand {\pdt}{\partial_t}
\newcommand {\pdtt}{\partial_{tt}}

\newcommand {\I}{[0,T]}
\newcommand {\Iopen}{(0,T)}

\newcommand {\domain}{\Omega}

\newcommand {\lp}[1]{\Leb^{#1} (\domain )}

\DeclareMathOperator{\Leb}{L}

\DeclareMathOperator{\Cont}{C}
\DeclareMathOperator{\Hi}{H}

\newcommand {\hk}[1]{\Hi^{#1}(\domain )}
\newcommand {\hko}[1]{\Hi^{#1}_0(\domain )}

\newcommand {\lpkIX}[2]{\Leb^{#1}\left(\Iopen,#2\right)}

\newcommand {\cIX}[1]{\Cont\left(\I,#1\right)}

\newcommand {\ckIX}[2]{\Cont^{#1}\left(\I,#2\right)}

\def \vector#1{\mathbf{#1}}

\newcommand {\X}{\vector{x}}

\usepackage{comment}
\usepackage[shortlabels]{enumitem}

\usepackage{algorithm}
\usepackage{algpseudocode}

\usepackage{subfigure}

\makeatletter
\@namedef{subjclassname@2020}{%
  \textup{2020} Mathematics Subject Classification}
\makeatother

\usepackage{xcolor}
\usepackage[hidelinks]{hyperref}
\hypersetup{
    colorlinks = true,
    linkcolor = red,
    anchorcolor = red,
    citecolor = green,
    filecolor = red,
    urlcolor = red,
    linktoc=page
    }
\usepackage{cleveref}

\begin{document}

	\title[...]{Rothe's method in direct and time-dependent inverse source problems for a semilinear pseudo-parabolic equation}
	

	\author[K.~Van~Bockstal]{Karel Van Bockstal$^1$} 
	\thanks{This work was supported by the grant no. AP23486218
 of the Ministry of  Science and High Education of  Republic
of Kazakhstan (MES RK), the Methusalem programme of Ghent University Special Research Fund (BOF) (Grant Number 01M01021) and the FWO Senior Research Grant G083525N}

	\address[1]{Ghent Analysis \& PDE center, Department of Mathematics: Analysis, Logic and Discrete Mathematics\\ Ghent University\\
		Krijgslaan 281\\ B 9000 Ghent\\ Belgium} 
	\email{karel.vanbockstal@UGent.be}

 \author[K.~Khompysh]{Khonatbek Khompysh$^{1,2,3}$} 

 \address[2]{Al-Farabi Kazakh National University\\ Almaty\\ Kazakhstan}
 \email{konat\_k@mail.ru}

  \author[A.~Altybay]{Arshyn Altybay$^{1,2,3}$} 
\address[3]{Institute of Mathematics and Mathematical Modeling,
\\ 125 Pushkin str., 050010 Almaty, Kazakhstan}
  \email{arshyn.altybay@gmail.com; arshyn.altybay@ugent.be}
	
	\begin{abstract} 
    In this paper, we investigate the inverse problem of determining an unknown time-dependent source term in a semilinear pseudo-parabolic equation with variable coefficients and a Dirichlet boundary condition. The unknown source term is recovered from additional measurement data expressed as a weighted spatial average of the solution. By employing Rothe's time-discretisation method, we prove the existence and uniqueness of a weak solution under a smallness condition on the problem data. We also provide a numerical scheme based on a perturbation approach, which reduces the solution of the resulting discrete problem to solving two standard variational problems and evaluating a scalar coefficient, and we demonstrate its accuracy and stability through numerical experiments.
	\end{abstract}
	
	\maketitle
	
	\tableofcontents

\section{Introduction}
\label{sec:introduction}

We consider an open and bounded Lipschitz domain $\Omega \subset \mathbb{R}^d$ with boundary $\Gamma$,  $d \in \NN$. 
Let $Q_T = (0,T) \times \Omega$ and $\Sigma_T = (0,T) \times \Gamma$, where $T>0$ is a given final time.
In the sequel, we consider the following  semilinear pseudo-parabolic equation with variable coefficients:
\begin{equation} \label{eq:problem}
	\left\{
	\begin{array}{rl}
		\pdt \left(\rho(t,\X) u(t,\X)\right) -  \nabla \cdot \left( \eta(t,\X) \nabla \pdt u(t,\X)\right)  & \\ [4pt]  - \nabla \cdot \left( \kappa(t,\X) \nabla u (t,\X)\right) 
		= f(u(t,\X))+\D F(t,\X),  & (t,\X) \in Q_T, \\[4pt]
		u=0, & (t,\X) \in \Sigma_T, \\[4pt]
		u(0,\X) = \tilde{u}_0  &\X \in \Omega,
	\end{array}
	\right.
\end{equation}
where  the source $F$ is decomposed as 
\[F(t,\X) = p(t,\X)h(t). \]

In this paper, the time-dependent part $h(t)$ of the source is unknown and will be recovered from the integral measurement
\begin{equation} \label{eq:add:cond}
	\int_\Omega  u(t,\X) \omega(\X)\dX = m(t). 
\end{equation}
In the system (\ref{eq:problem}-\ref{eq:add:cond}), the coefficients $\rho, \eta, \kappa $ and  $\tilde{u}_0, p, \omega, m, f, F$ are given functions, whilst $u$ and $h$ are unknown and need to be determined.

Equations like \eqref{eq:problem} involving third-order mixed derivatives $u_{xxt}$ are usually called pseudo-parabolic equations; in some works, they are also referred to as Sobolev-type equations. Pseudo-parabolic equations can be used to describe many important physical processes, such as the unidirectional propagation of nonlinear long waves \cite{Ting:1963,BBM:1972}, the aggregation of populations \cite{Pad:2004}, fluid flow in fissured rock \cite{BZK:1960}, filtration in porous media \cite{6BER:1989}, 
the nonsteady flow of second-order fluids \cite{Hui:1968}, and the motion of non-Newtonian fluids \cite{AKS2011,zvy-2010}, among others.
Numerous works on linear and nonlinear pseudo-parabolic equations have been devoted to studying direct problems. 

The integral condition \eqref{eq:add:cond} models a weighted spatial measurement of the state variable $u$. 
We assume that the observation device produces an output obtained by integrating local measurements of $u(t,\cdot)$ over the domain $\Omega$ with a prescribed weight function $\omega$ such that $\omega|_{\Gamma}=0$. It reflects that the sensor is insensitive near the boundary, which is consistent with the homogeneous Dirichlet boundary condition and excludes boundary effects from the measurement. 
Physically, such a measurement may arise from a distributed sensing mechanism or an interior averaging device that records a bulk signal proportional to the weighted total of temperature, concentration, or energy density inside the domain. 
In this framework, the function $m(t)$ represents the recorded measurement signal at time $t$, and the condition \eqref{eq:add:cond} provides the additional information required to identify the unknown time-dependent source term $h(t)$.

The study of inverse problems for pseudo-parabolic equations began with the early result of Rundell and Colton \cite{R:1980} in 1980, where an inverse problem for identifying source terms in a linear pseudo-parabolic equation was studied from overspecified boundary data.
The study of inverse problems for partial differential equations, as it is well known, significantly depends on the form of the measurement data. 
In particular, for inverse problems involving parabolic equations, various types of measurements, such as terminal values, boundary Cauchy data, and nonlocal measurements, have been utilised to recover a missing parameter. 
For instance, in the inverse problems of recovering a time-dependent unknown parameter,  the measurements in the form of an integral over the space domain have attracted considerable interest. This is due to the physical motivation that measurement data in integral form can reduce random noise from the measurement process through the averaging properties of integration. However, apart from the works \cite{LyTa:2011,LyVe:2019}, studies on inverse problems for pseudo-parabolic equations with nonlocal measurement \eqref{eq:add:cond} remain relatively rare.
Nevertheless, there are many works on the study of inverse problems for pseudo-parabolic equations; we may refer to \cite{AsAt:1997,FeUr:2004,LyTa:2011,LyVe:2019,AntAA:2020,KhSh:2023,Khompysh2025}, and references therein. In particular, in \cite{AntAA:2020,KhSh:2023}, inverse problems for nonlinear pseudo-parabolic equations perturbed by $p$-Laplacian and nonlinear damping term have been investigated under the measurement in the specific integral form, i.e., for the specially chosen test function in the form $\omega := \omega - \Delta \omega$ in  \eqref{eq:add:cond}. The same measurement has been considered in \cite{Khompysh2025} for the pseudoparabolic equation with a memory term. However, this form of measurement is difficult to justify its physical meaning, though it is more suitable for mathematical study. 
Recently, Van Bockstal and Khompysh \cite{VanBockstal2025} considered the inverse problem (\ref{eq:problem}-\ref{eq:add:cond}) with the Neumann boundary condition and $\omega=1$. Taking advantage of the Neumann boundary condition and the integral measurement over the complete domain $\Omega$, the existence and
uniqueness of a weak solution has been established. 

The main purpose of this paper is to study the inverse source problem for equation \eqref{eq:problem} with the measurement given in the form of \eqref{sec:direct_problem}, which is important from a physical perspective, and to establish the existence and uniqueness of a weak solution applying Rothe’s method \cite{Kacur1985}.  
The additional term 
\[
\nabla \cdot \left( \eta(t,\X) \nabla \partial_t u \right)
\] 
in the pseudo-parabolic equation in \eqref{eq:problem} provides extra regularity for the solution of the direct problem. However, this term introduces a technical challenge in the analysis of the inverse problem (\ref{eq:problem}-\ref{sec:direct_problem}) under Dirichlet boundary conditions. In the case of Neumann boundary conditions \cite{VanBockstal2025}, if the measurement is taken as a spatial average over the whole domain $\Omega$, one can apply the divergence theorem to rewrite the pseudo-parabolic term in terms of the flux of $\partial_t u$ on the boundary. This allows the discrete problem at each time step to be decoupled using the solution at the previous timestep $u_{i-1}$ (calculating subsequently $h_i$ and $u_i$), as employed in previous investigations \cite{Slodicka2014jcam,Grimmonprez2014b,VanBockstal2017,VanBockstal2020,VanBockstal2022a,VanBockstal2022c}.  

Under Dirichlet boundary conditions, the boundary values of $\partial_t u$ are not directly available, and the measurement does not provide information on the full boundary flux. 
Moreover, the decoupling strategy from earlier works is not effective as the previous timestep's velocity $\delta u_{i-1}$ is not given at $i=1.$ In this contribution, we overcome this barrier by introducing a perturbation approach at the discrete level, in which $\delta u_i$ is used instead of $\delta u_{i-1}$, and $u_i$ is computed first, followed by $h_i$. This strategy leads to the existence and uniqueness of a solution to (\ref{eq:problem}-\ref{sec:direct_problem}) under an additional source identification condition \eqref{cond:uniqueness}.

The paper is organised as follows. In \Cref{sec:direct_problem}, we study the corresponding direct problem in detail, which forms the basis for the analysis of the inverse problem. The inverse problem is treated in \Cref{sec:inverse_problem}, where we provide its weak formulation and prove the existence and uniqueness of weak solutions to \eqref{eq:problem} using Rothe's method. In \Cref{sec:experiments}, we present a practical numerical scheme based on a perturbation approach, which reduces the solution of the discrete problem to solving two standard variational problems and evaluating a scalar coefficient. We also establish discrete-time error estimates for this scheme, and numerical experiments are provided to illustrate its accuracy, stability, and convergence, even in the presence of small levels of noise.

\section{Direct problem}
\label{sec:direct_problem}

In this section, we consider the problem \eqref{eq:problem} with $F$ given and discuss its well-posedness. We will employ Rothe's method to obtain the existence of a weak solution to the problem. The inverse problem studied in \Cref{sec:inverse_problem} of the paper will be transformed into a direct problem. For this reason, we first study the direct problem in detail. Afterwards, we explain the differences in analysing the inverse problem (\ref{eq:problem}-\ref{sec:direct_problem}) in  \Cref{sec:inverse_problem}.

We first summarise the assumptions on the data that we will use to show the well-posedness of problem \eqref{eq:problem}:
\begin{enumerate}[\textbf{AS DP}-(1),leftmargin=2.4cm] 
\item  \label{as:DP:rho} $\rho\in \Cont^1(\I,\lp{\infty})$ with
\[
\begin{cases}
	0<\ubar{\rho}_0\leq \rho(t,\X)\leq \ubar{\rho}_1<\infty,\\
	\abs{ \pdt \rho (t,\X)} \leq \ubar{\rho}^\prime_1<\infty. 
\end{cases}
\]
\item\label{as:DP:eta} $\eta \in \Cont(\I,\lp{\infty})$ with $  0 < \ubar{\eta}_0 \leq \eta(t,\X)\leq \ubar{\eta}_1<\infty.$  
\item \label{as:DP:kappa}  $\kappa \in \Cont^1(\I,\lp{\infty})$ with
\[
\begin{cases}
	0<\ubar{\kappa}_0\leq \kappa(t,\X)\leq \ubar{\kappa}_1<\infty,\\
	\abs{\pdt \kappa(t,\X)}\leq \ubar{\kappa}^\prime_1<\infty.
\end{cases}
\]
\item \label{as:DP:f} $f:\mathbb{R} \rightarrow \mathbb{R}$ is  Lipschitz continuous, i.e. 
\[
\hspace*{2.4cm} \abs{f(s_1)-f(s_2)}\le L_f \abs{s_1-s_2}, \quad \forall s_1,s_2\in \RR. 
\]
\item \label{as:DP:p}  $F\in \Leb^2\left( (0,T] , \lp{2} \right).$
\item \label{as:DP:u0} $\tilde{u}_0\in\hko{1}.$
\end{enumerate}

\begin{remark}
We denote the $\lp{2}$ inner product by $\scal{\cdot}{\cdot}$ and its associated norm by $\vnorma{\cdot} = \sqrt{\scal{\cdot}{\cdot}}$. 
\end{remark}

\begin{remark}
From \ref{as:DP:f} it follows that 
\[
\abs{f(s)} \le \abs{f(0)} + L_f \abs{s}, \quad \forall s\in \RR.
\]
Hence, we have for $u:\Iopen \to \lp{2}$ that 
\begin{equation}\label{eq:inequality_nonlinear_f}
	\vnorma{f(u(t))}^2 \le \ubar{L}_f \left(\vnorma{u(t)}^2 + 1\right),  \qquad t\in \Iopen,
\end{equation}
where $\ubar{L}_f:= 2 \max \left\{L_f^2, f(0)^2 \abs{\Omega} \right\}.$
\end{remark}

\begin{remark}
We will frequently use the Friedrichs inequality throughout our analysis. It states that there exists a constant $C_{\textup F} > 0$, depending only on the domain $\Omega$, such that
\begin{equation}\label{eq:friedirchs}
	\vnorma{u} \le \vnorma{u}_{\hk{1}} \le C_{\textup F} \, \vnorma{\nabla u}, 
	\qquad \forall u \in \hko{1}.
\end{equation}
The constant $C_{\textup F}$ is known as the Friedrichs (or Poincaré) constant.
\end{remark}

For convenience, the main symbols, bounds, and constants used throughout the paper related to the direct problem are summarised in \Cref{tab:direct_problem}. 

\begin{table}[H]
\centering
\caption{Main coefficients and data for the direct problem.}
\label{tab:direct_problem}
\begin{tabular}{ll}
	\hline
	\textbf{Symbol} & \textbf{Meaning / Bounds} \\
	\hline
	$\rho(t,\X)$ & Density coefficient, $\rho \in \Cont^1(\I,\lp{\infty})$, $0<\ubar{\rho}_0 \le \rho \le \ubar{\rho}_1$, $|\pdt \rho| \le \ubar{\rho}_1'$ \\
	$\eta(t,\X)$ & Pseudo-parabolic coefficient, $\eta \in \Cont(\I,\lp{\infty})$, $0<\ubar{\eta}_0 \le \eta \le \ubar{\eta}_1$ \\
	$\kappa(t,\X)$ & Diffusion coefficient, $\kappa \in \Cont^1(\I,\lp{\infty})$, $0<\ubar{\kappa}_0 \le \kappa \le \ubar{\kappa}_1$, $|\pdt \kappa| \le \ubar{\kappa}_1'$ \\
	$f:\RR \to \RR$ & Nonlinearity, Lipschitz continuous with constant $L_f$, $\ubar{L}_f:= 2 \max \left\{L_f^2, f(0)^2 \abs{\Omega} \right\}$ \\
	$\tilde{u}_0(\X)$ & Initial condition, $\tilde{u}_0 \in \hko{1}$ \\
	$F(t,\X)$ & Forcing term, $F \in \Leb^2\left( (0,T] , \lp{2} \right)$ \\
	$C_{\textup F}$ & Friedrichs' constant\\
	$U_n$, $\overline{U}_n$ & Rothe functions\\
	\hline
\end{tabular}
\end{table}

The variational formulation of the direct problem reads as: 
\medskip
\begin{center}
Find $u(t)\in \hko{1}$ with $\pdt u(t)\in \hko{1}$ such that for a.a. $t \in \Iopen$ and any $\varphi \in \hko{1}$ it holds that 
\begin{equation}
	\scal{\pdt (\rho(t)u(t))}{\varphi} + \scal{\eta(t) \nabla \pdt u(t)}{\nabla \varphi} +  \scal{\kappa(t) \nabla u(t)}{\nabla \varphi} 
	=\scal{f(u(t))}{\varphi}+\scal{F(t)}{\varphi} \label{eq:var_for_direct_problem}
\end{equation}
\end{center}
and $u(0,\X) = \tilde{u}_0$ a.e. in $\Omega.$ 
\medskip

We first discuss the uniqueness of a solution to \eqref{eq:var_for_direct_problem}. 

\begin{theorem}\label{thm:uniqueness_direct_problem}
There exists at most one solution $u$ to \eqref{eq:var_for_direct_problem} satisfying 
\begin{equation*}
	u \in \cIX{\hko{1}} \quad \text{and} \quad  \pdt u \in \lpkIX{2}{\hko{1}}.
\end{equation*}
\end{theorem}

\begin{proof}
Let $u_1$ and $u_2$ be two distinct solutions to the direct problem \eqref{eq:var_for_direct_problem} with the same data. Then it follows from \eqref{eq:var_for_direct_problem} that for $u=u_1-u_2$ it holds that
\begin{equation}
	\scal{\pdt (\rho(t)u(t))}{\varphi} + \scal{\eta(t) \nabla \pdt u(t)}{\nabla \varphi} +  \scal{\kappa(t) \nabla u(t)}{\nabla \varphi} 
	= \scal{f(u_1(t))-f(u_2(t))}{ \varphi}. \label{uni:direct_problem:var_for}
\end{equation}
Setting $\varphi=\pdt u(t)$ in \eqref{uni:direct_problem:var_for} gives 
\begin{multline} \label{uni:direct_problem:eq1}
	\D\int_\Omega  (\pdt \rho) u (\pdt u) \dX + 
	\D\int_\Omega  \rho | \pdt u|^2\dX +\D\int_\Omega  \eta \abs{\nabla \pdt u}^2 \dX  \\  + \frac{1}{2}\D\int_\Omega  \kappa \pdt \abs{\nabla u}^2 \dX 
	=\D\int_\Omega(f(u_1)-f(u_2))\pdt u\dX .
\end{multline}
Next, integrating \eqref{uni:direct_problem:eq1} over $t\in(0,s)\subset (0,T)$ and using 
\begin{equation*}
	\int_0^s   \int_\Omega  \kappa\pdt \abs{\nabla u}^2 \dX \dt = \int_\Omega  \kappa(s) \abs{\nabla u(s)}^2 \dX -   \int_0^s   \int_\Omega (\pdt \kappa) \abs{\nabla u}^2 \dX \dt, 
\end{equation*}
which follows by integration by parts and $u(0,\cdot) =0,$ and employing the assumptions \ref{as:DP:rho}, \ref{as:DP:eta} and \ref{as:DP:kappa}, we have 
\begin{multline}\label{uni:dir_prob:eq1}
	\ubar{\rho}_0\int_0^s \vnorma{\pdt u (t)}^2 \dt + \ubar{\eta}_0 \int_0^s \vnorma{\nabla \pdt u(t)}^2 \dt  + \frac{\ubar{\kappa}_0}{2} \vnorma{\nabla u(s)}^2 \\
	\leq  \frac{\ubar{\kappa}^\prime_1}{2} \int_0^s \vnorma{\nabla u(t)}^2 \dt+\int_0^s\D\int_\Omega(f(u_1)-f(u_2))\pdt u\dX \dt-\int_0^s\D\int_\Omega  (\pdt \rho) u (\pdt u) \dX\dt.
\end{multline}

Moreover, using the H\"{o}lder and  $\veps$-Young inequalities, the Friedrichs inequality 
\eqref{eq:friedirchs},
and the Lipschitz continuity of $f$, we deduce that 
\begin{equation*}
	\abs{-\int_0^s \int_\Omega  (\pdt \rho) u (\pdt u) \dX \dt}  \leq \frac{\ubar{\rho}_1^{\prime 2} C_{\text F}^2}{4\veps_0} \int_0^s \vnorma{\nabla u(t)}^2 \dt + \veps_0 \int_0^s \vnorma{\pdt u(t)}^2 \dt, 
\end{equation*}
\begin{equation*}
	\abs{\int_0^s \int_\Omega  (f(u_1)-f(u_2))\pdt u \dX \dt}  \leq \frac{L_f^2 C_{\text F}^2}{4\veps_1} \int_0^s \vnorma{\nabla u(t)}^2 \dt + \veps_1 \int_0^s \vnorma{\pdt u(t)}^2 \dt. 
\end{equation*}
Hence, setting $\veps_0=\veps_1=\frac{\ubar{\rho}_0}{4}$ and plugging into \eqref{uni:dir_prob:eq1}, we obtain  
\begin{multline*}
	\ubar{\rho}_0\int_0^s \vnorma{\pdt u (t)}^2 \dt + 2\ubar{\eta}_0 \int_0^s \vnorma{\nabla \pdt u(t)}^2 \dt  + \ubar{\kappa}_0 \vnorma{\nabla u(s)}^2 \\
	\leq \left( \ubar{\kappa}^\prime_1 + 2 \frac{ C_{\text F}^2}{\ubar{\rho}_0}(\ubar{\rho}_1^{\prime 2}+L_f^2) \right) \int_0^s \vnorma{\nabla u(t)}^2 \dt.
\end{multline*}
Omitting the first two terms on the LHS and applying Gr\"onwall's lemma to the result, we obtain  that 
\begin{equation*}
	\vnorma{\nabla u (s)}^2   = 0 \quad \text{ for all } s\in (0,T),
\end{equation*}
which implies that $u=0$ a.e. in $Q_T.$ 
\end{proof}

The existence of a weak solution to problem \eqref{eq:var_for_direct_problem} will be addressed by employing Rothe's method, see e.g. \cite{Kacur1985}. We start by dividing the time interval $[0, T]$  into $n \in \mathbb{N}$ equidistant subintervals $[t_{i-1},t_i]$ of length $\tau = T/n<1$, $i = 1,\ldots n$. Hence, $t_i = i \tau$ for $i = 0,1, \ldots,n$.  We consider for any function $z$ that
\[
z_i \approx z(t_i) \quad \text{ and } \quad \pdt z(t_i) \approx \delta z_i = \dfrac {z_i-z_{i-1}}{\tau},
\]
i.e. the backward Euler method is used to approximate the time derivatives at every time step $t_i$. Moreover, linearising the right-hand side of \eqref{eq:var_for_direct_problem} at time step $t_i$ by replacing $u_i$ with $u_{i-1}$, we get the following semi-implicit time-discrete problem at time $t=t_i$: 
\medskip
\begin{center}
Find $u_i \in  \hko{1}$ such that 
\begin{equation}
	\scal{\delta (\rho_i u_i)}{\varphi} + \scal{\eta_i \nabla \delta u_i}{\nabla \varphi}  +  \scal{\kappa_i \nabla u_i}{\nabla \varphi} 
	=\scal{f(u_{i-1})}{\varphi}+\scal{F_i}{\varphi}, \quad \forall\varphi \in \hko{1}, \label{eq:var_for_direct_problem:discrete}
\end{equation}
\end{center}
where 
\medskip
\[ u_{0}=\tilde{u}_0.\]
We have that \eqref{eq:var_for_direct_problem:discrete} is equivalent with solving 
\begin{equation}\label{eq:var_for_direct_problem:discrete2}
a_i(u_i,\varphi) = l_i(\varphi), \quad \forall \varphi \in \hko{1}, 
\end{equation}
where $a_i: \hko{1}\times \hko{1}\to \RR$ is given by 
\begin{equation} \label{eq:bil_form_a}
a_i(u,\varphi) := \frac{1}{\tau}\scal{\rho_{i}u}{\varphi}  + \frac{1}{\tau} \scal{\eta_i \nabla u}{\nabla \varphi}  +  \scal{\kappa_i \nabla u}{\nabla \varphi},  \quad i=1,\ldots,n,  
\end{equation}
and $l_i: \hko{1}\to \RR$ is given by 
\begin{equation}
l_i(\varphi) := \scal{f(u_{i-1})}{\varphi}+\scal{F_i}{\varphi} + \frac{1}{\tau}\scal{\rho_{i-1}u_{i-1}}{\varphi} + \frac{1}{\tau} \scal{\eta_i \nabla u_{i-1}}{\nabla \varphi}, \quad i=1,\ldots,n. 
\end{equation}

\begin{theorem}\label{thm:dp:discrete}
Let the assumptions \ref{as:DP:rho} until \ref{as:DP:u0} be satisfied. For any $i=1,\ldots,n$, there exists a unique $u_i\in\hko{1}$ solving \eqref{eq:var_for_direct_problem:discrete}. 
\end{theorem}

\begin{proof}
The bilinear form $a_i$ is $\hko{1}$-continuous as 
\begin{equation*}
	\abs{a_i(u,\varphi)} \le \left(  \frac{1}{\tau} (\ubar{\rho}_1 + \ubar{\eta}_1) + \ubar{\kappa}_1\right) \vnorma{u}_{\hk{1}} \vnorma{\varphi}_{\hk{1}}, \quad \forall u,\varphi \in \hko{1}. 
\end{equation*}
Moreover, the form $a_i$ is also $\hko{1}$-elliptic since
\begin{align*}
	a_i(u,u) & \ge\frac{\ubar{\rho}_0}{\tau} \vnorma{u}^2  +\frac{\ubar{\eta}_0}{\tau}\vnorma{\nabla u}^2+ \ubar{\kappa}_0 \vnorma{\nabla u}^2 \\
	&\ge \min\left\{\frac{\ubar{\rho}_0}{\tau} , \frac{\ubar{\eta}_0}{\tau} +\ubar{\kappa}_0\right\}  \vnorma{u}_{\hk{1}}^2, \quad \forall u \in \hko{1}.  
\end{align*}
The linear functional $l_i$ is bounded on $\hko{1}$ if $u_{i-1}\in\hk{1}$ and $F_i \in \lp{2}$ as
\begin{equation*}
	\abs{l_i(\varphi)} \le \left( \ubar{L}_f  \left( \vnorma{ u_{i-1}} +1\right)+\vnorma{F_i} + \frac{\ubar{\rho}_1}{\tau} \vnorma{u_{i-1}} + \frac{\ubar{\eta}_1}{\tau} \vnorma{\nabla u_{i-1}} \right) \vnorma{\varphi}_{\hk{1}}, \; \forall \varphi\in \hko{1}. 
\end{equation*}
Hence, starting from $\tilde{u}_0\in \hko{1}$, we recursively obtain (as the conditions of the Lax-Milgram lemma are satisfied) the existence and uniqueness of $u_i \in \hko{1}$ for $i=1,\ldots,n.$ 
\end{proof}

The next step is to derive suitable a priori (stability) estimates for $u_i$ and $\delta u_i.$

\begin{lemma}\label{direct_problem:a_priori_estimate}
Let the assumptions \ref{as:DP:rho} until \ref{as:DP:u0} be satisfied. Then there exist  positive constants $C$ and $\tau_0$ such that 
\[
\sum_{i=1}^j \vnorma{\delta u_i}_{\hk{1}}^2 \tau + \vnorma{u_j}_{\hk{1}}^2
+  \sum_{i=1}^j \vnorma{u_i - u_{i-1}}_{\hk{1}}^2 \le C
\]
for all $j=1,\ldots,n$ and any $\tau < \tau_0.$
\end{lemma}

\begin{proof}
We start by putting $\varphi = \delta u_i \tau$ in \eqref{eq:var_for_direct_problem:discrete} and summing up the result for $i=1,\ldots,j$ with $1\le j\le n$. We obtain that 
\begin{multline}\label{dir_problem:a_priori_estimate:eq1}
	\sum_{i=1}^j \scal{\delta (\rho_i u_i)}{\delta u_i}\tau + \sum_{i=1}^j\scal{\eta_i \nabla \delta u_i}{\nabla \delta  u_i}\tau  +  \sum_{i=1}^j \scal{\kappa_i \nabla u_i}{\nabla  \delta  u_i} \tau \\
	= \sum_{i=1}^j \scal{f(u_{i-1})}{\delta  u_i}\tau+\sum_{i=1}^j \scal{F_i}{\delta  u_i}\tau.
\end{multline}
Note that $\|\kappa_i-\kappa_{i-1}\|_{\lp{\infty}}\le \|\partial_t\kappa\|_{\cIX{\lp{\infty}}} \tau$, so that  $\|\delta\kappa_i\|_{\lp{\infty}}\le \ubar{\kappa}^\prime_1$. Similarly, we have $\|\delta\rho_i\|_{\lp{\infty}}\le \ubar{\rho}^\prime_1$.
We have that
\begin{multline*}
	\sum_{i=1}^j  \scal{\kappa_i \nabla u_i}{\nabla \delta u_i} \tau 
	= \half \scal{\kappa_j \nabla u_j}{\nabla u_j}
	- \half \scal{\kappa_0 \nabla \tilde{u}_0}{\nabla  \tilde{u}_0}\\
	- \half \sum_{i=1}^j \scal{(\delta\kappa_i) \nabla u_{i-1}}{\nabla u_{i-1}} \tau 
	+ \half \sum_{i=1}^j \scal{\kappa_i (\nabla u_i -  \nabla u_{i-1})}{\nabla u_i - \nabla u_{i-1}}.
\end{multline*}
Using \ref{as:DP:kappa}, we get as $\tau <1$ that
\begin{multline*}
	\sum_{i=1}^j  \scal{\kappa_i \nabla u_i}{\nabla \delta u_i} \tau\\ 
	\geq \frac{\ubar{\kappa}_0}{2}  \vnorma{\nabla u_j}^2 - \left( \frac{\ubar{\kappa}_1}{2} +  \frac{\ubar{\kappa}_1^\prime}{2}\right) \vnorma{\nabla \tilde{u}_0}^2
	- \frac{\ubar{\kappa}_1^\prime}{2}  \sum_{i=1}^{j-1} \vnorma{\nabla u_{i}}^2 \tau  
	+ \frac{\ubar{\kappa}_0}{2} \sum_{i=1}^j \vnorma{\nabla u_i - \nabla u_{i-1}}^2.  
\end{multline*}
Employing the rule
\begin{equation}\label{eq:delta_rule}
	\delta (ab)_i = \delta (a_ib_i) = (\delta a_i ) b_i + a_{i-1} (\delta b_i), \quad i \in \NN, 
\end{equation}
the $\veps$-Young inequality and the Friedrichs inequality \eqref{eq:friedirchs}, we have by \ref{as:DP:rho} that 
\begin{align*}
	\sum_{i=1}^j \scal{\delta (\rho_i u_i)}{\delta u_i}\tau 
	& \ge \sum_{i=1}^j \scal{\rho_{i-1} (\delta u_i)}{\delta u_i}\tau - \abs{\sum_{i=1}^j \scal{(\delta \rho_i) u_i}{\delta u_i}\tau } \\
	& \ge \left(\ubar{\rho}_0 - {\veps_1}  \right) \sum_{i=1}^j \vnorma{\delta u_i}^2 \tau - \frac{\ubar{\rho}_1^{\prime 2} C_{\textup{F}}^2}{4\veps_1} \sum_{i=1}^j \vnorma{\nabla u_i}^2 \tau.  
\end{align*}
Moreover, using \ref{as:DP:f} and the Friedrichs inequality \eqref{eq:friedirchs}, we obtain that 
\[
\abs{\sum_{i=1}^j \scal{f(u_{i-1})}{\delta  u_i}\tau} 
\le
\frac{\ubar{L}_f}{4\veps_2} \left(T + C_{\textup{F}}^2 \vnorma{\nabla \tilde u_0}^2 \right) + \frac{ \ubar{L}_fC_{\textup{F}}^2}{4\veps_2} \sum_{i=1}^{j-1}  \vnorma{\nabla {u_{i}}}^2 \tau  
+   \veps_2\sum_{i=1}^j \vnorma{\delta u_i}^2 \tau
\]
and 
\begin{equation*}
	\abs{\sum_{i=1}^j \scal{F_i}{\delta  u_i}\tau}  \le
	\frac{1}{4\veps_3} \sum_{i=1}^j\vnorma{F_i}^2 \tau +   \veps_3\sum_{i=1}^j \vnorma{\delta u_i}^2 \tau. 
\end{equation*}
Hence, choosing $\veps_1=\veps_2 = \veps_3=\frac{\ubar{\rho}_0}{6}$, we obtain from \eqref{dir_problem:a_priori_estimate:eq1} that 
\begin{multline*}
	\ubar{\rho}_0 \sum_{i=1}^j \vnorma{\delta u_i}^2 \tau + 2\ubar{\eta}_0 \sum_{i=1}^j \vnorma{\nabla \delta u_i}^2 \tau 
	+ \ubar{\kappa}_0  \vnorma{\nabla u_j}^2 +  \ubar{\kappa}_0 \sum_{i=1}^j \vnorma{\nabla u_i - \nabla u_{i-1}}^2 \\ 
	\le  \frac{3}{\ubar{\rho}_0} \left(\ubar{L}_f  T + \left( \ubar{\kappa}_1+  \ubar{\kappa}_1^\prime + \ubar{L}_f C_{\textup{F}}^2 \right) \vnorma{\nabla \tilde{u}_0}^2  + \sum_{i=1}^j\vnorma{F_i}^2 \tau \right) \\
	+ \left(\ubar{\kappa}_1^\prime +  3 \left(\ubar{\rho}_1^{\prime 2} +  \ubar{L}_f\right) \frac{C_{\textup{F}}^2}{\ubar{\rho}_0} \right)\sum_{i=1}^j \vnorma{\nabla u_i}^2 \tau. 
\end{multline*}
Then, for $\tau$ sufficiently small, the proof concludes by applying the Gr\"onwall lemma.  
\end{proof}

In the next step, we introduce the so-called Rothe functions: The piecewise linear in-time function
\begin{equation}
\label{Rothestepfun:u}U_n:[0,T]\to\lp{2}:t\mapsto\begin{cases} \tilde{u}_0 & t =0,\\
	u_{i-1} + (t-t_{i-1})\delta u_i &
	t\in (t_{i-1},t_i],\quad 1\leqslant i\leqslant n, \end{cases}\end{equation}
	and piecewise constant function
	\begin{equation}
\label{Rothefun:u}
\overline U_n:[-\tau,T] \to\lp{2}:t\mapsto\begin{cases} \tilde{u}_0 &  t \in [-\tau,0],\\
	u_i &     t\in (t_{i-1},t_i],\quad 1\leqslant i\leqslant n.\end{cases}
	\end{equation}
	Similarly, in connection with the given functions $\rho$, $\pdt \rho,$ $\eta$, $\kappa$ and $F,$ we define the functions $\overline{\rho}_n$, $\overline{\pdt \rho}_n$, $\overline{\eta}_n$, $\overline{\kappa}_n$  and $\overline{F}_n,$ respectively. Using these functions and \eqref{eq:delta_rule}, we rewrite the discrete variational formulation \eqref{eq:var_for_direct_problem:discrete} as follows ($t\in(0,T]$)
	\begin{multline}\label{eq:var_for_direct_problem:whole_time_frame}
\scal{\overline{\rho}_n(t-\tau) \pdt U_n(t) +  \overline{\pdt \rho}_n(t) \overline U_n(t)}{\varphi} + \scal{\overline{\eta}_n(t) \nabla \pdt U_n(t)}{\nabla \varphi} \\ +  \scal{\overline{\kappa}_n(t) \nabla \overline U_n(t)}{\nabla \varphi}
=\scal{f(\overline{U}_n(t-\tau))}{\varphi}+\scal{\overline{F}_n(t)}{\varphi}, \quad \forall\varphi \in \hko{1}.
\end{multline}

We show the existence of a solution in the next theorem. 

\begin{theorem}\label{thm:existence_direct_problem}
Let the assumptions \ref{as:DP:rho} until \ref{as:DP:u0} be satisfied. A unique weak solution $u$ exists to \eqref{eq:var_for_direct_problem} satisfying 
\begin{equation*}
	u \in \cIX{\hko{1}} \quad \text{and} \quad  \pdt u \in \lpkIX{2}{\hko{1}}.
\end{equation*}
\end{theorem}

\begin{proof}
We have from \Cref{direct_problem:a_priori_estimate} that there exist $C>0$ and $n_0\in \NN$ such that for all $n\geqslant n_0 > 0$ it holds that 
\[
\max_{t\in [0,T]} \vnorma{\overline{U}_n(t)}_{\hk{1}}^2  + 
\int_0^T \vnorma{\pdt {U}_n(t)}^2 \dt+\sum\limits_{i=1}^n\vnorma{\ \int\limits_{t_{i-1}}^{t_i}\pdt \nabla {U}_n(t) \dt}^2\leqslant C.
\]
The compact embedding  $\hko{1} \imbed \imbed \lp{2} $ (see \cite[Theorem~6.6-3]{Ciarlet2013}) or \cite[Lemma~1.3.13]{Kacur1985}) leads to the existence of a function
$u \in \cIX{\lp{2}}  \cap \Leb^{\infty}\left((0,T), \hko{1}\right)$ with $\pdt u \in \lpkIX{2}{\lp{2}}$,
and a subsequence $\{ U_{n_l}\}_{l\in\NN}$ of $\{U_n\}$ such that
\begin{equation} \label{convergence:rothe_functions_dp}
	\left\{
	\begin{array}{ll}
		U_{n_l} \to u & \text{in}~~\Cont\left([0,T], \lp{2}\right), \\[4pt]
		U_{n_l}(t) \rightharpoonup u(t) & \text{in}~~\hko{1},~~\forall t \in [0,T], \\[4pt]
		\overline{U}_{n_l}(t) \rightharpoonup u(t) & \text{in}~~\hko{1},~~\forall t \in [0,T],  \\[4pt]
		\pdt {U}_{n_l} \rightharpoonup \pdt u & \text{in}~~\Leb^{2}\left((0,T), \lp{2}\right).
	\end{array}
	\right.
\end{equation}
Additionally, from \Cref{direct_problem:a_priori_estimate}, we obtain that 
\begin{equation*}
	\int_0^T \vnorma{\pdt {U}_{n_l}(t)}_{\hk{1}}^2 \dt \leqslant C. 
\end{equation*}
By the reflexivity of the space $\lpkIX{2}{\hko{1}}$, we have
that
\begin{equation}\label{weak_convergence_time_der}
	\pdt {U}_{n_l} \rightharpoonup \pdt u \quad  \text{in} \quad \Leb^{2}\left((0,T), \hko{1}\right), 
\end{equation}
i.e. $u\in\cIX{\hko{1}}$.
Moreover, from \Cref{direct_problem:a_priori_estimate}, we also have that (note that $\tau_l = T/n_l$)
\begin{equation}\label{dir:str:conv:L2}
	\int_0^{T} \left( \vnorma{\overline{U}_{n_l}(t) - U_{n_l}(t) }^2 + \vnorma{\overline{U}_{n_l}(t-\tau) - U_{n_l}(t) }^2 \right) \dt  \leqslant 2 \tau_l^2 \sum_{i=1}^{n_l} \vnorma{\delta u_i}^2 \tau  \leqslant C \tau_l^2,
\end{equation}
so \begin{equation}\label{dir:str:conv:L22}\overline U_{N_l},  \overline U_{N_l}(\cdot-\tau) \to u \ \text{in} \ \lpkIX{2}{\lp{2}} \ \text{as} \ l\rightarrow \infty.\end{equation}
Now, we integrate \eqref{eq:var_for_direct_problem:whole_time_frame} for $n=n_l$ over $t\in(0,\eta)\subset \Iopen$ to get that 
\begin{multline}\label{eq:var_for_direct_problem:whole_time_frame_int}
	\int_0^\eta \scal{\overline{\rho}_{n_l}(t-\tau) \pdt U_{n_l}(t) +  \overline{\pdt \rho}_{n_l}(t) \overline U_{n_l}(t)}{\varphi} \dt  + \int_0^\eta \scal{\overline{\eta}_{n_l}(t) \nabla \pdt U_{n_l}(t)}{\nabla \varphi} \dt \\ +  \int_0^\eta \scal{\overline{\kappa}_{n_l}(t) \nabla \overline U_{n_l}(t)}{\nabla \varphi} \dt
	= \int_0^\eta\scal{f(\overline{U}_{n_l}(t-\tau))}{\varphi}\dt+\int_0^\eta \scal{\overline{F}_{n_l}(t)}{\varphi} \dt, \quad \forall\varphi \in \hko{1}. 
\end{multline}
We now pass to the limit \(l\to\infty\) in
\eqref{eq:var_for_direct_problem:whole_time_frame_int}.
The convergence of the linear terms follows from
\eqref{convergence:rothe_functions_dp}--\eqref{weak_convergence_time_der},
together with the strong convergence \eqref{dir:str:conv:L22}, and the boundedness and continuity assumptions
on the coefficients \(\rho\), \(\eta\), and \(\kappa\).
For the nonlinear term, the Lipschitz continuity of \(f\)
(Assumption~\ref{as:DP:f}) combined with the strong convergence
\(\overline U_{n_l}(\cdot-\tau) \to u\) in \(\lpkIX{2}{\lp{2}}\)
yields
\[
f(\overline U_{n_l}(\cdot-\tau)) \to f(u)
\quad \text{ in } \lpkIX{2}{\lp{2}}.
\]
Moreover, since \(\overline F_{n_l} \to F\) in \(\lpkIX{2}{\lp{2}}\),
we may also pass to the limit in the forcing term.
Therefore, passing \(l\to\infty\) in
\eqref{eq:var_for_direct_problem:whole_time_frame_int} gives
\begin{multline*}
	\int_0^\eta \scal{\rho(t) \pdt u(t) +  \pdt \rho(t) u(t)}{\varphi} \dt  + \int_0^\eta \scal{\eta(t) \nabla \pdt u(t)}{\nabla \varphi} \dt \\ +  \int_0^\eta \scal{\kappa(t) \nabla u(t)}{\nabla \varphi} \dt
	= \int_0^\eta\scal{f({u}(t))}{\varphi}\dt+\int_0^\eta \scal{F(t)}{\varphi} \dt, \quad \forall\varphi \in \hko{1}. 
\end{multline*}
Differentiating this result with respect to $\eta$ implies that $u$ solves \eqref{eq:var_for_direct_problem}. The uniqueness of a solution (see \Cref{thm:uniqueness_direct_problem}) gives that the entire Rothe sequence $\{U_n\}$ converges in $\cIX{\lp{2}}$ to the solution. 
\end{proof}

\section{Inverse source problem}
\label{sec:inverse_problem}


In this section, we will show the existence and uniqueness of a weak solution to the inverse source problem (\ref{eq:problem}-\ref{sec:direct_problem}).
Next to assumptions \ref{as:DP:rho} until \ref{as:DP:u0}, we will need the following additional assumptions
\begin{enumerate}[\textbf{AS DP}-(\arabic*),leftmargin=2.4cm]  
\setcounter{enumi}{6}
\item \label{as:rho_extra} $\rho \in \cIX{\hk{1}}.$
\item \label{as:omega}  $ \omega\in \hko{1}$.
\item \label{as:p} $p\in \cIX{\lp{2}}$ and the function $\ubar{\omega}:[0,T]\to \mathbb{R}$ defined  by 
\[\ubar{\omega}(t):=\scal{p(t)}{\frac{\omega}{\rho(t)}}\]
satisfies
\[
\ubar{\omega}(t) \neq 0 \text{ for all } t\in \I \ \ \text{and}  \ \ \ubar{\omega}\in \Cont\left(\I\right).
\]
We denote 
\[
0< \ubar{\omega}_m:=\min_{t\in \I} \abs{\ubar{\omega}(t)}.
\]
\item \label{as:m}  $m\in \Hi^1((0,T])$.
\item \label{as:uniqueness} There exists a positive constant $\zeta<1$ such that  
\begin{equation} \label{cond:uniqueness}
	\vnorma{p}_{\mathcal{X}}^2  \frac{ M^2  \ubar{\eta}_1^2}{4\ubar{\omega}_m^2 \ubar{\eta}_0 \ubar{\rho}_0}  < \zeta, 
\end{equation}
where  $M:=\sup\limits_{t\in[0,T]} \vnorma{\nabla \left(\dfrac{\omega}{\rho(t)}\right)}$ and $\mathcal{X} := \cIX{\lp{2}}$. 
\end{enumerate}

\begin{remark}[About the condition \eqref{cond:uniqueness}]
The uniqueness and existence of a solution will rely on the technical condition \eqref{cond:uniqueness}. This condition can be interpreted in several ways:
\begin{itemize}
	\item As a smallness condition on the spatial source profile $ p $ as it is evidently satisfied if the norm $ \|p\|_{\mathcal{X}} $ is sufficiently small;
	\item As a condition on the problem's parameters: The condition also involves the upper and lower bounds of the coefficient  $ \eta $ and the lower bound of $ \rho $. Specifically, it requires the ratio $ \frac{\ubar{\eta}_1^2}{\ubar{\eta}_0 \ubar{\rho}_0} $ to be sufficiently small.
	\item Role of the measurement weight $ \omega $. The constant $ M $, which depends on the gradient of $ \omega/\rho $, and $ \bar{\omega}_m $ also influence the condition. A well-chosen weight function $ \omega $ that is sufficiently ``smooth'' (small $M$) and provides a strong signal (large $ \bar{\omega}_m $) makes the condition easier to satisfy.
\end{itemize}
The search for a proof technique that does not require this, or that relaxes it to a more natural condition, remains an open and interesting problem for future work.
\end{remark}

\begin{remark}
The identifiability condition 
$(p(t), \tfrac{\omega}{\rho(t)}) \neq 0$
in \ref{as:p} requires that the spatial weighting
function $\omega$ used in the measurement is not orthogonal to the spatial source
profile $p(t)$ (up to the known scaling $\rho(t)$).
Note that overlap of the supports of $p(t)$ and $\omega$ is not sufficient, since sign changes may lead to cancellation in the integral.
Identifiability is therefore guaranteed if $\omega$ is chosen to be strictly positive
on a subset of $\Omega$ where $p(t)$ does not change sign.
From an experimental design perspective, this can be enforced by placing sensors in regions where the source profile has a
dominant sign.
\end{remark}

\begin{remark}[Instability of the inverse source problem]
Inverse source problems for parabolic and pseudo-parabolic equations are well known
to be ill-posed in the sense of Hadamard: small perturbations in the data may lead to
large deviations in the reconstructed source.

This instability can already be illustrated in a simple linear pseudo-parabolic
setting.
Consider the one-dimensional problem
\[
\partial_t u - \partial_{xx} u - \partial_{xx}\partial_t u
= h(t)p(x), \quad x \in (0,\pi),
\]
with homogeneous Dirichlet boundary conditions and zero initial data.
Assume that the measurement is given by
\[
m(t) = \int_0^\pi u(t,x)\omega(x)dx ,
\qquad \omega \in \Hi_0^1(0,\pi).
\]
Expanding the solution in the eigenfunctions $\sin(kx)$ of $-\partial_{xx}$ (with corresponding eigenvalues $\lambda_k = k^2$) yields
\[
u(t,x) = \sum_{k=1}^\infty u_k(t)\sin(kx),
\]
where the coefficients $u_k(t)$ satisfy
\[
(1+k^2)u_k^\prime(t) + k^2 u_k(t) = h(t)p_k, \quad u_k(0) =0,
\]
with $
p_k := \frac{2}{\pi} \int_0^\pi p(x) \sin(kx)\dx.$
Solving this equation gives
\[
u_k(t)
= \frac{p_k}{1+k^2}
\int_0^t e^{-\frac{k^2}{1+k^2}(t-s)} h(s)\ds .
\]
Consequently, the measurement can be written as
\[
m(t)
= \sum_{k=1}^\infty
\frac{p_k\omega_k}{1+k^2}
\int_0^t e^{-\frac{k^2}{1+k^2}(t-s)} h(s)\ds, 
\]
where $
\omega_k := \int_0^\pi \omega(x) \sin(kx)\dx.$ This is a first-kind Volterra integral equation for $h$, which is highly sensitive to perturbations in the data: since the kernel
\(\frac{p_k \, \omega_k}{1 + \eta k^2} e^{-\frac{k^2}{1 + \eta k^2} (t-s)}\)
decays rapidly for large \(k\), even small errors in $m(t)$ can be strongly amplified when reconstructing $h(t)$.
\end{remark}

Next to the main notations for the direct problem in \Cref{tab:direct_problem}, we summarise the one for the inverse problem in \Cref{tab:inverse_problem}.

\begin{table}[htbp]
\centering
\caption{Main quantities and parameters for the inverse problem.}
\label{tab:inverse_problem}
\begin{tabular}{ll}
	\hline
	\textbf{Symbol} & \textbf{Meaning / Bounds} \\
	\hline
	$h(t)$ & Unknown source, $h\in \Leb^2(0,T)$ \\
	$p(t,\X)$ & Given part of the source $ph$, $p \in \cIX{\lp{2}}$ \\
	$\omega(\X)$ & Spatial weighting function, $\omega \in \hko{1}$ \\
	$m(t)$ & Measurement, $m \in \Hi^1((0,T])$ \\
	$\ubar{\omega}(t)$ & Measurement scalar, $\ubar{\omega}(t) = \scal{p(t)}{\frac{\omega}{\rho(t)}} \neq 0$, $\ubar{\omega}_m := \min_{t \in \I} |\ubar{\omega}(t)|>0$ \\
	$M$ & Gradient bound, $M = \sup_{t \in \I} \vnorma{\nabla (\omega/\rho(t))}$ \\
	$\zeta$ & Uniqueness parameter, $0<\zeta<1$, satisfying \eqref{cond:uniqueness} \\
	\hline
\end{tabular}
\end{table}

Now, we will derive an expression for $h$ in terms of $u$ and the given data. In this way, we will be able to reformulate the inverse problem as a coupled direct problem.
We assume that the weight function $\omega$ satisfies \ref{as:omega} such that the derivatives of the solution can be transferred to the weight function $\omega$, making it possible to calculate the necessary estimates when proving the uniqueness and existence of a solution to the inverse problem. For this reason, we multiply \cref{eq:problem} by $\frac{\omega}{\rho}$ and integrate over $\Omega$. Using 
$\pdt (\rho u)= (\pdt \rho) u + \rho \pdt u$ and Green's theorem, as $\omega\in \hko{1}$, we obtain  that 
\begin{multline} \label{eq:expression_h}
h(t)=\frac{1}{\ubar{\omega}(t)}\left[ m^\prime(t) +  \scal{u(t)}{\frac{\omega \pdt \rho(t)}{\rho(t)}} \right. \\ \left.
+\scal{\eta(t) \nabla \pdt u (t) }{\nabla \left(\frac{\omega}{\rho(t)}\right)} + \scal{\kappa(t)\nabla u(t)}{\nabla \left(\frac{\omega}{\rho(t)}\right)}   - \scal{f(u(t))}{\frac{\omega}{\rho(t)}}\right], 
\end{multline}
where we have used \ref{as:p}.
Hence, the following variational formulation can be defined: 
\medskip
\begin{center}
Find $u(t)\in \hko{1}$ with $\pdt u(t)\in \hko{1}$ such that for a.a. $t \in \Iopen$ and any $\varphi \in \hko{1}$ it holds that 
\begin{equation}
	\scal{\pdt (\rho(t)u(t))}{\varphi} + \scal{\eta(t) \nabla \pdt u(t)}{\nabla \varphi} +  \scal{\kappa(t) \nabla u(t)}{\nabla \varphi} 
	=h(t) \scal{p(t)}{\varphi} + \scal{f(u(t))}{\varphi} , \label{eq:var_for}
\end{equation}
with $h(t)$ given by \eqref{eq:expression_h}.
\end{center}
\medskip

Compared to previous inverse source problem connected to parabolic equations wherein the time-dependent part of the source is of interest, see e.g. \cite{Slodicka2014jcam,VanBockstal2022c} the additional $\nabla \cdot \left( \eta(t,\X) \nabla \pdt u\right)$ term in the pseudo-parabolic equation in \eqref{eq:problem} complicates the analysis and leads to the additional restriction \eqref{cond:uniqueness} on the data in our analysis. We start with showing the uniqueness of a solution to the problem (\ref{eq:expression_h}-\ref{eq:var_for}) under this condition.

\begin{theorem}\label{thm:uniq_inv_problem}
Let the assumptions \ref{as:DP:rho} until \ref{as:uniqueness} be fulfilled. 
Then, there exists at most one couple $\{u,h\}$ solving problem (\ref{eq:expression_h}-\ref{eq:var_for})  such that
\begin{equation*}
	h \in \Leb^2\Iopen, \quad  u \in \cIX{\hko{1}} \quad \text{with} \quad  \pdt u \in \lpkIX{2}{\hko{1}}.
\end{equation*}
\end{theorem}

\begin{proof}
Let $u_1$ and $u_2$ be two distinct solutions of the inverse problem (\ref{eq:expression_h}-\ref{eq:var_for}) with the same data. 
We subtract the variational formulation for $(u_1, h_1)$ from the one for $(u_2,h_2)$. Then, we obtain  for $u=u_1-u_2$ and $h=h_1-h_2$  that
\begin{multline}
	\scal{\pdt (\rho(t)u(t))}{\varphi} + \scal{\eta(t) \nabla \pdt u(t)}{\nabla \varphi} +  \scal{\kappa(t) \nabla u(t)}{\nabla \varphi} \\
	=\scal{f(u_1(t))-f(u_2(t))}{\varphi} +h(t) \scal{p(t)}{\varphi}, \quad \forall \varphi \in \hko{1},  \label{uni:var_for}
\end{multline}
and 
\begin{multline} \label{uni:expression_h}
	h(t)=
	\frac{1}{\ubar{\omega}(t)}\left[ \scal{u(t)}{\frac{\omega \pdt \rho(t)}{\rho(t)}}+\scal{\eta(t) \nabla \pdt u(t) }{\nabla \left(\frac{\omega}{\rho(t)}\right)} \right. \\ \left.+ \scal{\kappa(t)\nabla u(t)}{\nabla \left(\frac{\omega}{\rho(t)}\right)} -\scal{f(u_1(t))-f(u_2(t))}{\frac{\omega}{\rho(t)}}  \right].
\end{multline}
Using H\"{o}lder's inequality and \ref{as:p}, we first estimate \eqref{uni:expression_h} by taking its absolute value:
\begin{multline*} 
	\abs{h(t)} \leq \frac{1}{\abs{\ubar{\omega}(t)}} \left[ \vnorma{u(t)} \vnorma{\frac{\omega \pdt \rho(t)}{\rho(t)}}
	+ \vnorma{\eta(t) \nabla \pdt u(t)} \vnorma{\nabla \left(\frac{\omega}{\rho(t)}\right)} \right.\\ \left.
	+ \vnorma{\kappa(t)\nabla u(t)} \vnorma{\nabla \left(\frac{\omega}{\rho(t)}\right)}+ \vnorma{f(u_1(t))-f(u_2(t))} \vnorma{\frac{\omega}{\rho(t)}}\right].
\end{multline*}
Using \ref{as:DP:rho}, \ref{as:DP:eta}, \ref{as:DP:kappa}, \ref{as:DP:f}, \ref{as:omega} and \ref{as:p},  we obtain that 
\begin{equation*}
	\abs{h(t)} \le \frac{1}{\ubar{\omega}_m} \left[\frac{\ubar{\rho}_1^\prime}{\ubar{\rho}_0} \vnorma{u(t)}\vnorma{\omega} + M\left( \ubar{\eta}_1 \vnorma{\nabla \pdt u(t)} + \ubar{\kappa}_1 \vnorma{\nabla u(t)}\right) + \frac{L_f}{\ubar{\rho}_0} \vnorma{u(t)} \vnorma{\omega}\right].
\end{equation*}
Hence, we get by the Friedrichs inequality \eqref{eq:friedirchs} that 
\begin{equation}\label{IP:uni:estimate_h}
	\abs{h(t)} \le 
	C_1 \vnorma{\nabla u(t)} + \frac{ M  \ubar{\eta}_1}{\ubar{\omega}_m}  \vnorma{\nabla \pdt u(t)},
\end{equation}
where 
\[
C_1:= \frac{1}{\ubar{\omega}_m}  \left[\frac{\vnorma{\omega}C_{\textup F}}{\ubar{\rho}_0} \left(\ubar{\rho}_1^\prime + L_f \right) + M\ubar{\kappa}_1 \right].
\]
Now, we take $\varphi=\pdt u(t)$ in \eqref{uni:var_for} and  integrate the result over $t\in(0,s)\subset (0,T)$ to get 
\begin{multline} \label{00:eq:est2}
	\D\int_0^s \int_\Omega  \rho | \pdt u|^2\dX\dt +\D\int_0^s \int_\Omega  \eta \abs{\nabla \pdt u}^2 \dX\dt   + \frac{1}{2}\D\int_0^s \int_\Omega  \kappa \pdt \abs{\nabla u}^2 \dX \dt \\
	=\D\int_0^s \int_\Omega(f(u_1)-f(u_2))\pdt{u}\dX\dt +\int_0^s h(t) \D\int_\Omega   p \pdt{u}\dX \dt- \int_0^s \int_\Omega  (\pdt \rho) u (\pdt u) \dX \dt.
\end{multline}
The terms on the left-hand side and the first term on the right-hand side can be handled as in \Cref{thm:uniqueness_direct_problem}. 
For the second term on the right-hand side of \eqref{00:eq:est2}, we obtain that
\begin{align*}
	& \abs{\int_0^s h(t) \D \left(\int_\Omega   p \pdt{u}\dX \right)\dt} \\
	&\leqslant \int_0^s |h(t)| \vnorma{p(t)} \vnorma{\pdt u(t)} \dt \\
	& \leqslant \vnorma{p}_{\mathcal{X}} \int_0^s \left(C_1 \vnorma{\nabla u(t)} + \frac{ M  \ubar{\eta}_1}{\ubar{\omega}_m}  \vnorma{\nabla \pdt u(t)}\right)\vnorma{\pdt u(t)} \dt \\
	& \leqslant \vnorma{p}_{\mathcal{X}}^2 \frac{C_1^2}{4\veps_2} \int_0^s \vnorma{\nabla u(t)}^2\dt +  \vnorma{p}_{\mathcal{X}}^2  \frac{ M^2  \ubar{\eta}_1^2}{4\ubar{\omega}_m^2 \veps_3} \int_0^s \vnorma{\nabla \pdt u(t)}^2\dt \\
	& \qquad +  (\veps_2+\veps_3) \int_0^s \vnorma{\pdt u(t)}^2 \dt.
\end{align*}
Collecting all estimates, we have that 
\[
\alpha\int_0^s \vnorma{\pdt u}^2 \dt + \beta \int_0^s \vnorma{\nabla \pdt u}^2 \dt  + \ubar{\kappa}_0\vnorma{\nabla u(s)}^2   \leqslant  C_2\int_0^s \vnorma{\nabla u(t)}^2 \dt, 
\]
where 
\begin{align*}
	\alpha &:=2 \left( \ubar{\rho}_0 - \sum_{i=0}^3  \veps_i \right), \\
	\beta &:= 2 \left( \ubar{\eta}_0 -  \vnorma{p}_{\mathcal{X}}^2  \frac{ M^2  \ubar{\eta}_1^2}{4\ubar{\omega}_m^2 \veps_3}  \right), \\
	C_2 &:=  \ubar{\kappa}^\prime_1 + \frac{ \ubar{\rho}_1^{\prime 2} C_{\textup F}^2}{2 \veps_0}+\frac{L_f^2C_{\textup F}^2}{2\veps_1} + \vnorma{p}_{\mathcal{X}}^2 \frac{C_1^2}{2\veps_2}. 
\end{align*}
Choosing $\veps_3 = \ubar{\rho}_0 \zeta$, and using condition
\eqref{cond:uniqueness} in \ref{as:uniqueness}, we observe that $\beta > 0$.
Next, choosing $\veps_0=\veps_1=\veps_2=\frac{\ubar{\rho}_0(1-\zeta)}{6}$, we also have that $\alpha=  \ubar{\rho}_0(1-\zeta) > 0.$ Therefore, applying the Gr\"onwall lemma gives that 
\[
\int_0^s \vnorma{\pdt u(t)}^2 \dt +  \int_0^s \vnorma{\nabla \pdt u(t)}^2 \dt  + \vnorma{\nabla u(s)}^2 = 0, 
\]
from which we conclude that $u=0$ a.e. in $Q_T.$ Moreover, from \eqref{IP:uni:estimate_h}, we obtain that $h=0$ a.e. in $\Iopen.$
\end{proof}

Next, we will show the existence of a weak solution to problem (\ref{eq:expression_h}-\ref{eq:var_for}) by employing Rothe's method. The time-discrete weak formulation of the inverse problem (\ref{eq:expression_h}-\ref{eq:var_for}) reads as:
\medskip
\begin{center}
Find $u_i \in \hko{1}$ and $h_i\in \RR$ such that 
\begin{equation} \label{eq:disc_inv_prob}
	\scal{\delta (\rho_i u_i)}{\varphi} + \scal{\eta_i \nabla \delta u_i}{\nabla \varphi}  +  \scal{\kappa_i \nabla u_i}{\nabla \varphi} 
	=h_{i}\scal{p_i}{\varphi}+\scal{f(u_{i-1})}{\varphi}, \quad \forall\varphi \in \hko{1}, 
\end{equation}
and
\begin{multline} \label{disc:hi-1}
	h_{i}=\frac{1}{\ubar{\omega}_i}\left[ m^\prime_i +  \scal{u_{i-1}}{\frac{\omega \delta \rho_i}{\rho_i}} \right. \\ \left.
	+\scal{\eta_i \nabla \delta u_{i}}{\nabla \left(\frac{\omega}{\rho_i}\right)} + \scal{\kappa_i\nabla u_{i-1}}{\nabla \left(\frac{\omega}{\rho_i}\right)}   -\scal{f(u_{i-1})}{\frac{\omega}{\rho_i}}\right], 
\end{multline}
where
\begin{equation}
	u_{0}=\tilde{u}_0. \label{initC:disc_inv_prob}
\end{equation}
\end{center}
\medskip

Usually, the approach is to decouple the inverse problem at the discrete time steps (scheme: $h_i \rightarrow u_i$). Here, we consider $\scal{\eta_i \nabla \delta u_{i}}{\nabla \left(\frac{\omega}{\rho_i}\right)}$ instead of  $\scal{\eta_i \nabla \delta u_{i-1}}{\nabla \left(\frac{\omega}{\rho_i}\right)}$ in \eqref{disc:hi-1}, in order to obtain a uniformly defined scheme for
all $i \ge 1$. Indeed, the quantity
$\nabla \delta u_{i-1} = \tau^{-1}(\nabla u_{i-1} - \nabla u_{i-2})$
is not available at the first time step $i=1$.
So, for given $i\in \{1,\ldots,n\},$  we first solve \eqref{eq:disc_inv_prob} for $u_i$ and then derive $h_i$ from \eqref{disc:hi-1}. Consequently, solving \eqref{disc:hi-1} is equivalent with solving 
\begin{equation}\label{equiv:var_for_inv_disc_prob}
b_i(u_i,\varphi) = g_i(\varphi), \quad \forall \varphi \in \hko{1}, 
\end{equation}
where the bilinear form  $b_i: \hko{1}\times \hko{1}\to \RR$ is given by 
\begin{equation*}
b_i(u,\varphi) := \frac{1}{\tau}\scal{ \rho_i u}{\varphi} +  \frac{1}{\tau} \scal{\eta_i \nabla u}{\nabla \varphi}  +  \scal{\kappa_i \nabla u}{\nabla \varphi} - 
\frac{1}{\tau\ubar{\omega}_i}\scal{\eta_i \nabla  u}{\nabla \left(\frac{\omega}{\rho_i}\right)} \scal{p_i}{\varphi}
\end{equation*}
and the linear functional $g_i: \hko{1}\to \RR$ is defined by 
\begin{multline}
g_i(\varphi) := \frac{1}{\ubar{\omega}_i}\left[m^\prime_i +  
\scal{u_{i-1}}{\frac{\omega \delta \rho_i}{\rho_i}} 
- \frac{1}{\tau}\scal{\eta_i \nabla u_{i-1}}{\nabla \left(\frac{\omega}{\rho_i}\right)} \right. \\ \left.
+ \scal{\kappa_i\nabla u_{i-1}}{\nabla \left(\frac{\omega}{\rho_i}\right)}   -\scal{f(u_{i-1})}{\frac{\omega}{\rho_i}}
\right] \scal{p_i}{\varphi}+\\
\scal{f(u_{i-1})}{\varphi} + \frac{1}{\tau}\scal{\rho_{i-1}u_{i-1}}{\varphi} + \frac{1}{\tau} \scal{\eta_i \nabla u_{i-1}}{\nabla \varphi}. 
\end{multline}

The discrete inverse problem leads to an elliptic problem with a perturbation. It is condition \ref{as:uniqueness} that guarantees that this perturbation does not destroy coercivity, as illustrated in the next theorem. First, the differences between the previous strategies and the current Dirichlet strategy are summarised in \Cref{tab:discrete_strategies}.

\begin{table}[h!]
\centering
\caption{Comparison of discrete-time strategies for the inverse problem.}
\label{tab:discrete_strategies}
\begin{tabular}{|l|l|}
	\hline
	Reference & Strategy / Remarks \\ \hline
	\cite{Slodicka2014jcam,Grimmonprez2014b,VanBockstal2022a,VanBockstal2025} & $h_i \rightarrow u_i$:  $u_{i-1}$ available; decoupling straightforward \\ \hline
	\cite{VanBockstal2017,VanBockstal2020,VanBockstal2022c} & $h_i \rightarrow u_i$: time-discrete decoupling uses $\delta u_{i-1}$ (available) and $u_{i-1}$ \\ \hline
	This work & $u_i \rightarrow h_i$: $\delta u_{i-1}$ not available, perturbation approach ensuring coercivity under \ref{as:uniqueness}\\ \hline
\end{tabular}
\end{table}

\begin{theorem} \label{thm:IP_disc_existence}
Let the conditions \ref{as:DP:rho} until \ref{as:uniqueness} be fulfilled.  Then, for any $i=1,\ldots,n$, there exists a unique couple $\{u_i, h_i\}\in\hko{1}\times \RR$ solving \eqref{eq:disc_inv_prob}-\eqref{disc:hi-1}. 
\end{theorem}

\begin{proof}
We only check that the bilinear form $b_i$ is $\hko{1}$-elliptic for any $i\in \{1,\ldots,n\}$. Applying the $\veps$-Young inequality with $\veps = \ubar{\rho}_0 \zeta$ (where $\zeta$ is defined by \ref{as:uniqueness}),  we have 
\begin{align*}
	\abs{\frac{1}{\ubar{\omega}_i}\scal{\eta_i \nabla  u}{\nabla \left(\frac{\omega}{\rho_i}\right)} \scal{p_i}{u}} 
	&\le  \frac{1}{\ubar{\omega}_m} \ubar{\eta}_1 \vnorma{\nabla u} M \vnorma{p}_{\mathcal{X}} \vnorma{u} \\
	& \le \ubar{\rho}_0 \zeta \vnorma{u}^2 + \frac{1}{4\ubar{\rho}_0 \zeta} \frac{\ubar{\eta}_1^2  M^2 \vnorma{p}_{\mathcal{X}}^2 }{\ubar{\omega}_m^2} \vnorma{\nabla u}^2,
\end{align*}
where $M=\sup\limits_{t\in[0,T]} \vnorma{\nabla \left(\dfrac{\omega}{\rho(t)}\right)}.$
Hence, we have for all $u\in \hko{1}$ and $i\in \{1,\ldots,n\}$ that 
\begin{align*}
	b_i(u,u) & \ge \frac{\ubar{\rho}_0}{\tau} \vnorma{u}^2+\frac{\ubar{\eta}_0}{\tau} \vnorma{\nabla u}^2   + \ubar{\kappa}_0 \vnorma{\nabla u}^2 - \abs{\frac{1}{\tau\ubar{\omega}_i}\scal{\eta_i \nabla  u}{\nabla \left(\frac{\omega}{\rho_i}\right)} \scal{p_i}{u}}  \\
	& \ge \frac{\ubar{\rho}_0}{\tau} \left(1-\zeta\right)\vnorma{u}^2+ \frac{1}{\tau} \left(\ubar{\eta}_0 -  \frac{\ubar{\eta}_1^2  M^2 \vnorma{p}_{\mathcal{X}}^2 }{4\ubar{\rho}_0 \zeta\ubar{\omega}_m^2} \right) \vnorma{\nabla u}^2   + \ubar{\kappa}_0 \vnorma{\nabla u}^2 \\
	& \stackrel{\eqref{eq:friedirchs}}{\ge} \frac{\ubar{\kappa}_0}{C_{\textup{F}}^2} \vnorma{u}_{\hk{1}}^2. 
\end{align*}
Therefore, starting from $\tilde{u}_0\in \hko{1},$   we get the existence of $u_i \in \hko{1}$ for any  $i=1,\ldots,n,$ since all conditions of the Lax-Milgram lemma are satisfied. Moreover, $h_i\in \RR$ is uniquely determined by \eqref{disc:hi-1}. 
\end{proof}

As for the direct problem, we now derive the a priori estimates for the inverse problem.

\begin{lemma}\label{inv_prob:_est1}
Let the assumptions \ref{as:DP:rho} until \ref{as:uniqueness} be fulfilled. Then, there exists positive constants $C$ and $\tau_0$ such that 
\begin{equation}\label{est:invprob_discrsolu}
	\max\limits_{1\leq j\leq n}\vnorma{u_j}_{\hk{1}}^2
	+ \sum_{i=1}^n \vnorma{\delta u_i}_{\hk{1}}^2 \tau +  \sum_{i=1}^n \vnorma{u_i - u_{i-1}}_{\hk{1}}^2+ \sum_{i=1}^n |h_i|^2\tau \le C
\end{equation}
for any $\tau < \tau_0.$
\end{lemma}

\begin{proof}
Setting $\varphi = \delta u_i \tau$ in \eqref{eq:disc_inv_prob} and summing up the result for $i=1,\ldots,j$ with $1\le j\le n$ give
\begin{multline}\label{direct_problem:a_priori_estimate:eq1}
	\sum_{i=1}^j \scal{\delta (\rho_i u_i)}{\delta u_i}\tau + \sum_{i=1}^j\scal{\eta_i \nabla \delta u_i}{\nabla \delta  u_i}\tau  +  \sum_{i=1}^j \scal{\kappa_i \nabla u_i}{\nabla  \delta  u_i} \tau\\
	=\sum_{i=1}^j h_{i}\scal{p_i}{\delta  u_i}\tau+\sum_{i=1}^j \scal{f(u_{i-1})}{\delta  u_i}\tau
\end{multline}
We follow the lines of \Cref{direct_problem:a_priori_estimate}. The difference lies in estimating the first term on the right-hand side of \eqref{direct_problem:a_priori_estimate:eq1}. In doing so, we first estimate $h_i$ given by \eqref{disc:hi-1}. 
We obtain
\begin{multline*}
	|h_{i}|\leq\frac{1}{\abs{\ubar{\omega}_i}}\left[\abs{ m^\prime_i} +  \frac{\ubar{\rho}_1^\prime}{\ubar{\rho}_0}\vnorma{u_{i-1}}\vnorma{\omega} +\ubar{\eta}_1 \vnorma{\nabla \delta u_{i}}\vnorma{\nabla \left(\frac{\omega}{\rho_i}\right)} \right.\\
	\left. + \ubar{\kappa}_1\vnorma{\nabla u_{i-1}}\vnorma{\nabla \left(\frac{\omega}{\rho_i}\right)}  +\frac{1}{\ubar{\rho}_0}\vnorma{f(u_{i-1})}\vnorma{\omega}\right]. 
\end{multline*}
Using the Friedrichs inequality \eqref{eq:friedirchs}, we get that 
\begin{equation}\label{ineq:hi}
	|h_{i}|\leq\frac{1}{{\ubar{\omega}_m}}\abs{ m^\prime_i} +  C_1\vnorma{\nabla u_{i-1}}
	+\frac{\ubar{\eta}_1 M}{\ubar{\omega}_m} \vnorma{\nabla \delta u_{i}}
	+\frac{1}{\ubar{\rho}_0\ubar{\omega}_m}\vnorma{f(u_{i-1})}\vnorma{\omega}, 
\end{equation}
where 
\[
C_1:=\frac{1}{{\ubar{\omega}_m}}\left(\frac{\ubar{\rho}_1^\prime C_{\textup{F}}}{\ubar{\rho}_0}\vnorma{\omega}  + M \ubar{\kappa}_1\right), \quad M=\sup\limits_{t\in[0,T]} \vnorma{\nabla \left(\dfrac{\omega}{\rho(t)}\right)}.
\]
Employing the $\veps$-Young inequality, we obtain that
\begin{align*}
	\left|\sum_{i=1}^j h_{i} \scal{p_i}{\delta  u_i}\tau\right|
	& \leq  \sum_{i=1}^j \left(\vnorma{p}_{{\mathcal X}} |h_{i}|\sqrt{\tau}\right)\left(\vnorma{\delta  u_i}\sqrt{\tau}\right) \\
	&  \leq \frac{\vnorma{p}_{\mathcal{X}}^2}{{\ubar{\omega}_m^2}4\veps_3}   \sum_{i=1}^j \abs{ m^\prime_i}^2\tau +  \frac{ \vnorma{p}_{\mathcal{X}}^2 C_1^2}{4\veps_4} \sum_{i=1}^j \vnorma{\nabla u_{i-1}}^2 \tau \\
	& \quad +\frac{\vnorma{p}_{\mathcal{X}}^2\ubar{\eta}_1^2 M^2}{\ubar{\omega}_m^2 4 \veps_5} \sum_{i=1}^j \vnorma{\nabla \delta u_{i}}^2 \tau 
	+\frac{\vnorma{p}_{\mathcal{X}}^2\vnorma{\omega}^2}{\ubar{\rho}_0^2\ubar{\omega}_m^2 4 \veps_6} \sum_{i=1}^j \vnorma{f(u_{i-1})}^2 \tau \\
	& \quad + \left(\sum_{i=3}^6 \veps_i \right) \sum_{i=1}^j  \vnorma{\delta u_i}^2 \tau.
\end{align*}
Therefore, using \eqref{eq:inequality_nonlinear_f}, we get the estimate
\begin{align*}
	\left|\sum_{i=1}^j h_{i} \scal{p_i}{\delta  u_i}\tau\right|
	&  \leq \frac{\vnorma{p}_{\mathcal{X}}^2}{{\ubar{\omega}_m^2}4\veps_3}   \sum_{i=1}^j \abs{ m^\prime_i}^2\tau 
	+  \frac{ \vnorma{p}_{\mathcal{X}}^2 C_1^2}{4\veps_4} \left( \vnorma{\nabla \tilde{u}_0}^2 +  \sum_{i=1}^{j-1} \vnorma{\nabla u_{i}}^2 \tau\right) \\
	&  +\frac{\vnorma{p}_{\mathcal{X}}^2\vnorma{\omega}^2}{\ubar{\rho}_0^2\ubar{\omega}_m^2 4 \veps_6} \ubar{L}_f \left(T + C_{\textup{F}}^2\vnorma{\nabla \tilde{u}_{0}}^2 +  C_{\textup{F}}^2 \sum_{i=1}^{j-1} \vnorma{\nabla u_{i}}^2 \tau \right)  \\
	&  + \left(\sum_{i=3}^6 \veps_i \right) \sum_{i=1}^j  \vnorma{\delta u_i}^2 \tau  +\frac{\vnorma{p}_{\mathcal{X}}^2\ubar{\eta}_1^2 M^2}{\ubar{\omega}_m^2 4 \veps_5} \sum_{i=1}^{j} \vnorma{\nabla \delta u_{i}}^2 \tau.
\end{align*}
Hence, using the estimates obtained in  \Cref{direct_problem:a_priori_estimate} (with $F=0$), we obtain 
\begin{multline*}
	\alpha \sum_{i=1}^j \vnorma{\delta u_i}^2 \tau + \beta \sum_{i=1}^j \vnorma{\nabla \delta u_i}^2 \tau 
	+ \ubar{\kappa}_0  \vnorma{\nabla u_j}^2 +  \ubar{\kappa}_0 \sum_{i=1}^j \vnorma{\nabla u_i - \nabla u_{i-1}}^2 \\ 
	\le C_0 +  C_2 \sum_{i=1}^j \vnorma{\nabla u_i}^2 \tau,
\end{multline*}
where 
\begin{align*}
	\alpha &:=2 \left( \ubar{\rho}_0 - \sum_{i=1}^6  \veps_i \right), \\
	\beta &:= 2 \left( \ubar{\eta}_0 -  \vnorma{p}_{\mathcal{X}}^2  \frac{ M^2  \ubar{\eta}_1^2}{4\ubar{\omega}_m^2 \veps_5}  \right), \\
	C_0& : =  \left(\frac{1}{2\veps_2} + \frac{\vnorma{p}_{\mathcal{X}}^2\vnorma{\omega}^2}{\ubar{\rho}_0^2\ubar{\omega}_m^2 2 \veps_6}  \right) \ubar{L}_f  T 
	+ \frac{\vnorma{p}_{\mathcal{X}}^2}{{\ubar{\omega}_m^2}2\veps_3} \sum_{i=1}^j \abs{ m^\prime_i}^2\tau  \\
	&   \qquad 
	+ \left( \ubar{\kappa}_1+  \ubar{\kappa}_1^\prime + \frac{\ubar{L}_f C_{\textup{F}}^2}{2\veps_2} + \frac{ \vnorma{p}_{\mathcal{X}}^2 C_1^2}{2\veps_4} + \frac{\vnorma{p}_{\mathcal{X}}^2\vnorma{\omega}^2\ubar{L}_f C_{\textup{F}}^2}{\ubar{\rho}_0^2\ubar{\omega}_m^2 2 \veps_6}  \right) \vnorma{\nabla \tilde{u}_0}^2, \\
	C_2 &:=    \ubar{\kappa}^\prime_1 + \frac{ \ubar{\rho}_1^{\prime 2} C_{\textup F}^2}{2 \veps_1}+\frac{L_f^2C_{\textup F}^2}{2\veps_2}  + \frac{ \vnorma{p}_{\mathcal{X}}^2 C_1^2}{2\veps_4} + \frac{\vnorma{p}_{\mathcal{X}}^2\vnorma{\omega}^2\ubar{L}_f C_{\textup{F}}^2}{\ubar{\rho}_0^2\ubar{\omega}_m^2 2 \veps_6}.
\end{align*}
Now, putting  $\veps_5 = \ubar{\rho}_0 \zeta$ implies that $\beta >0$ due to the condition \eqref{cond:uniqueness} in \ref{as:uniqueness}. 
Next, choosing $\veps_0=\veps_1=\veps_2 = \veps_3 = \veps_4=\veps_6=\frac{\ubar{\rho}_0}{5\ubar{n}}$ with $\ubar{n}>\frac{1}{1-\zeta}> 1$ (as $\zeta<1$), we also have that $\alpha = \ubar{\rho}_0 \left(1-\zeta - \frac{1}{\ubar{n}}\right) > 0.$ Therefore, applying the Gr\"onwall lemma gives the existence of a positive constant $C_3$ such that 
\[
\sum_{i=1}^j \vnorma{\delta u_i}^2 \tau +\sum_{i=1}^j \vnorma{\nabla \delta u_i}^2 \tau 
+\vnorma{\nabla u_j}^2 +  \sum_{i=1}^j \vnorma{\nabla u_i - \nabla u_{i-1}}^2 \le C_3
\]
for sufficiently small $\tau \le \tau_0.$ Moreover, from \eqref{ineq:hi} and \ref{as:m}, we obtain that 
\[
\sum_{i=1}^j \abs{h_i}^2 \tau \le C_4. 
\]
\end{proof}

Now, using the Rothe's functions \eqref{Rothestepfun:u} and \eqref{Rothefun:u},
we rewrite the discrete variational formulation (\ref{eq:disc_inv_prob}-\ref{disc:hi-1}) as follows ($t\in(0,T]$)
\begin{multline}\label{eq:disc_inv_prob:whole_time_frame}
\scal{\overline{\rho}_n(t-\tau) \pdt U_n(t) +  \pdt \rho_n(t) \overline U_n(t)}{\varphi} + \scal{\overline{\eta}_n(t) \nabla \pdt U_n(t)}{\nabla \varphi} \\ +  \scal{\overline{\kappa}_n(t) \nabla \overline U_n(t)}{\nabla \varphi} 
={\overline{h}_n(t)}\scal{\overline{p}_n(t)}{\varphi}+\scal{{f}(\overline{U}_n(t-\tau))}{\varphi}, \quad \forall\varphi \in \hko{1},
\end{multline}
and
\begin{multline} \label{disc:hi-1:whole_time_frame}
\overline{h}_n(t)=\frac{1}{\overline{\ubar{\omega}}_{n}(t)}\left[ {\overline{m^\prime}_n(t)}+  \scal{\overline{U}_{n}(t-\tau)}{\left(\frac{\omega {\pdt \rho}_n(t)}{\overline{\rho}_n(t)}\right)} +\right. \\ \left.
\scal{\overline{\eta}_n(t) {\nabla \pdt U_{n}}(t) }{\nabla \left(\frac{\omega}{\overline{\rho}_n(t)}\right)} + \scal{\overline{\kappa}_n(t)\nabla \overline{U}_n(t-\tau)}{\nabla \left(\frac{\omega}{\overline{\rho}_n(t)}\right)} \right. \\ \left.   -\scal{f(\overline{U}_{n}(t-\tau))}{\frac{\omega}{\overline{\rho}_n(t)}} \right]. 
\end{multline}

Now, we are ready to show the existence of a solution to (\ref{eq:expression_h}-\ref{eq:var_for}).

\begin{theorem}\label{thm:existence_inverse_problem}
Let the assumptions \ref{as:DP:rho} until \ref{as:uniqueness} be fulfilled.  Then, there exists a unique weak solution couple $\{u,h\}$ to (\ref{eq:expression_h}-\ref{eq:var_for}) satisfying 
\begin{equation*}
	u \in \cIX{\hko{1}} \quad \text{with} \quad  \pdt u \in \lpkIX{2}{\hko{1}} \quad \text{and} \quad h \in \Leb^2\Iopen.
\end{equation*}
\end{theorem}

\begin{proof}
From \Cref{inv_prob:_est1}, we obtain that the convergences stated in \eqref{convergence:rothe_functions_dp}, \eqref{weak_convergence_time_der} and \eqref{dir:str:conv:L22} are valid. Moreover, we have that 
\[
\vnorma{\overline{h}_{n}}_{\Leb^2(0,T)}^2\le C, 
\]
i.e. there exists a function $g\in \Leb^2\Iopen$ such that $\overline{h}_{n_l} \rightharpoonup g$ in $\Leb^2\Iopen.$ We easily see that 
\[
\int_0^T \overline{h}_{n_l}(t) \phi(t)\dt \rightarrow \int_0^T {h}(t) \phi(t)\dt \quad \text{ for all } \phi \in \Leb^2\Iopen \text { as } l \to \infty, 
\]
where $h$ is given by \eqref{eq:expression_h}. Hence, by the uniqueness of the weak limit, we have that $g=h.$ Next, we integrate \eqref{eq:disc_inv_prob:whole_time_frame} for $n=n_l$ over $t\in(0,\eta)\subset \Iopen$ and pass to the limit $l\to \infty.$ Afterwards, we differentiate with respect to $\eta$ to obtain that \eqref{eq:var_for} is satisfied. The uniqueness of a solution and the convergence of the whole Rothe sequences follow from \Cref{thm:uniq_inv_problem}. 
\end{proof} 

\begin{remark}[Semilinear heat equation]
The approach proposed in this paper for tackling the inverse source problem (\ref{eq:problem}-\ref{sec:direct_problem}) remains valid for the semilinear anisotropic heat equation (i.e., \cref{eq:problem} with $\eta=0$). In this context, we also refer to \cite[Remark~1.2]{VanBockstal2022c} in which a similar remark was made for isotropic heat conduction (as a special case of the isotropic thermoelastic system). However, in this paper, the regularity of the temperature (which becomes $\pdt u \in \lpkIX{2}{\lp{2}}$ for $\eta=0$)  has been handled more adequately compared to \cite[Theorem~4.1]{VanBockstal2022c}. 
\end{remark}

\section{Numerical experiments}
\label{sec:experiments}

In this section, we solve the inverse
problem (\ref{eq:problem}-\ref{eq:add:cond}) numerically by the aid of the finite element method.
First, we introduce in \Cref{subsec:reformulation_discrete_problem} an approach for solving the discrete inverse problem \eqref{equiv:var_for_inv_disc_prob} and summarise the strategy in \Cref{algorithm}. 
Afterwards, we present in \Cref{subsec:noisy_free} a convergence test for noise-free data in one spatial dimension. Next, in \Cref{subsec:1dnoisy}, we perform 
noisy experiments in one dimension.
We note that all computations were carried out with the FEniCSx platform
\cite{BarattaEtal2023,ScroggsEtal2022,BasixJoss,AlnaesEtal2014}, based on version \texttt{0.1} of the
DOLFINx module.

In the experiments, we consider $\{\Omega = (0,1), T=1\}$ and the exact solutions
\begin{equation} \label{eq:exact_solution1}
u_1(t,x) = \exp(-t) \sin(\pi x),
\qquad
h_1(t) = \exp(-t),
\end{equation}
and
\begin{equation} \label{eq:exact_solution2}
u_2(t,x) = \cos(2\pi t)\sin(\pi x),
\qquad
h_2(t) = \sin(2\pi t).
\end{equation}
The spatial domain $\Omega = (0,1)$ is discretised by a uniform mesh consisting of $400$ subintervals. The  finite element space is chosen as the space of continuous, piecewise linear functions corresponding to first-order Lagrange elements.

Next, we define the weight function $\omega$ used in the experiments. 
Let $y\in(0,1)$, $\delta>0$ and $\bar{\varepsilon}\in(0,\delta)$.
We define the weight function $\omega:(0,1)\to\mathbb{R}$ by
\begin{equation} \label{eq:def_mollifier}
\omega(x)
=
\begin{cases}
	1,
	& |x-y|\le \delta-\bar{\varepsilon}, \\[0.3em]
	\dfrac12\!\left(
	1+\cos\!\Big(
	\dfrac{\pi(|x-y|-\delta+\bar{\varepsilon})}{2\bar{\varepsilon}}
	\Big)
	\right),
	& \delta-\bar{\varepsilon} < |x-y| < \delta+\bar{\varepsilon}, \\[0.6em]
	0,
	& |x-y|\ge \delta+\bar{\varepsilon} .
\end{cases}
\end{equation}
The function $\omega$ is nonnegative, compactly supported in
$(y-\delta-\bar{\varepsilon},y+\delta+\bar{\varepsilon})$, and belongs to
$\omega\in \Cont^\infty_0(0,1)$.
Moreover, $\omega\equiv1$ in a neighborhood of $y$ and vanishes smoothly
at the boundary of its support.
The shape of the weight function $\omega$ is illustrated in
\Cref{fig:mollifier} for $y=0.5$, $\delta=0.1$ and $\bar{\varepsilon}=0.01$. It is precisely this shape of the weight function that will be used in all numerical experiments.

In all experiments, we take $f=0$, $p(x)= 1 + \sin(\pi x)$, $\rho =1$, $\kappa(t) = t+1$ and $\eta =0.001.$ This choice of $\eta$ makes that the smallness condition in \ref{as:uniqueness} is satisfied as
\[
\vnorma{p}_{\mathcal{X}}^2  \frac{ M^2  \ubar{\eta}_1^2}{4\ubar{\omega}_m^2 \ubar{\eta}_0 \ubar{\rho}_0}  = \frac{\eta \vnorma{\nabla \omega}^2 \vnorma{p}^2}{4 \scal{p}{\omega}^2 \rho} \approx 0.54 <1, 
\]
since $M = \frac{\vnorma{\nabla \omega}}{\rho}.$

\begin{figure}[ht]
\centering
\includegraphics[width=0.7\textwidth]{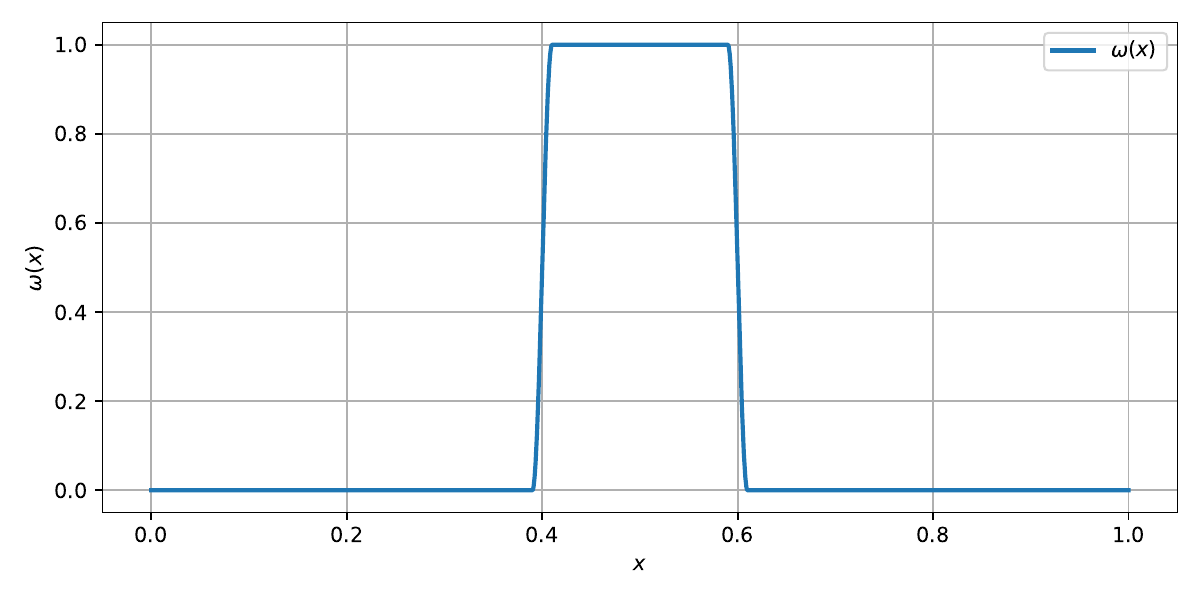}
\caption{Weight function $\omega$ defined in \eqref{eq:def_mollifier}
	for $y=0.5$, $\delta=0.1$ and $\bar{\varepsilon}=0.01$.}
\label{fig:mollifier}
\end{figure}

The reconstruction strategy described in \Cref{algorithm} requires access to the
time derivative $m^\prime(t)$ of the measurement. Since numerical differentiation
of noisy data is inherently unstable, this step is regularised by approximating
the noisy measurements with a low-degree polynomial.
Starting from the exact signal $m(t)$, noisy observations $m^{\epsilon}(t)$ are
generated according to
\[
m^{\epsilon}(t)
=
m(t)\bigl(1+\epsilon \mathcal{R}(t)\bigr),
\qquad
t \in \left\{ \frac{jT}{\widetilde{N}} : j=0,\ldots,\widetilde{N}=100 \right\},
\]
where $\epsilon>0$ denotes the relative noise level (e.g.\ $\epsilon=0.01$
corresponds to $1\%$ noise), and $\mathcal{R}(t)$ is a sequence of independent random
variables uniformly distributed in $[-1,1]$. To ensure reproducibility of the
numerical experiments, the random number generator is initialised with a fixed
seed $\texttt{np.random.seed}(\texttt{exp}-1)$, where $\texttt{exp}$ denotes the
experiment index (so either 1 or 2).

For each noise level $\epsilon \in \{0.001,0.005,0.01,0.03,0.05\}$, the noisy
data $m^\epsilon(t)$ are approximated by polynomials $p^{\epsilon}_{\tilde{d}}(t)$
of degree $\tilde{d} \le 10$, computed via a linear least-squares fit using
\texttt{np.polyfit}. The polynomial degree is selected based on the relative
reduction of the discrete $\ell^2$-error when increasing the degree, defined by
\begin{equation}\label{eq:definition_improvememt}
r_{\mathrm{im}}(\tilde{d})
=
100 \times
\frac{\|p^\epsilon_{\tilde{d}-1}-m^\epsilon\|_{\ell^2}
	-
	\|p^\epsilon_{\tilde{d}}-m^\epsilon\|_{\ell^2}}
{\|p^\epsilon_{\tilde{d}-1}-m^\epsilon\|_{\ell^2}},
\end{equation}
where
\[
\|f\|_{\ell^2}
=
\left(
\frac{1}{\widetilde{N}+1}
\sum_{i=0}^{\widetilde{N}} |f(t_i)|^2
\right)^{1/2}.
\]
The values of $r_{\mathrm{im}}(\tilde{d})$ corresponding to the exact solutions
\eqref{eq:exact_solution1} and \eqref{eq:exact_solution2} are reported in
\cite[Tables~1 and~2]{VanBockstal2025}. Based on this analysis, we select a
polynomial of degree $\tilde{d}=3$ for \eqref{eq:exact_solution1} and
$\tilde{d}=8$ for \eqref{eq:exact_solution2} for each noise level $\epsilon$. Consequently, whenever $\epsilon\neq 0$, the derivative $m^\prime(t_i)$ used in the
reconstruction step \eqref{disc:hi-1} is approximated by
\[
m^\prime(t_i)
\approx
\begin{cases}
\bigl(p^{\epsilon}_{3}\bigr)'(t_i),
& \text{for the experiment corresponding to \eqref{eq:exact_solution1}},\\[0.4em]
\bigl(p^{\epsilon}_{8}\bigr)'(t_i),
& \text{for the experiment corresponding to \eqref{eq:exact_solution2}},
\end{cases}
\qquad i=1,\ldots,n.
\]

\subsection{Reformulation discrete problem}
\label{subsec:reformulation_discrete_problem}

In this section, we present the approach we followed for solving the discrete inverse problem \eqref{equiv:var_for_inv_disc_prob} at a fixed time level $t_i$,
which reads as
\begin{equation}\label{equiv:var_for_inv_disc_prob2}
b_i(u_i,\varphi) = g_i(\varphi),
\qquad \forall \varphi \in \hko{1}.
\end{equation}
First, we define the linear functional
\begin{equation}\label{eq:B_i_def}
B_i(u)
:= \scal{\eta_i \nabla u}{\nabla\!\left(\frac{\omega}{\rho_i}\right)},
\qquad u \in \hko{1}. 
\end{equation}
Then, the bilinear form $b_i(\cdot,\cdot)$ appearing in
\eqref{equiv:var_for_inv_disc_prob2} can be written as
\begin{equation}\label{eq:b_rank_one}
b_i(u,\varphi)
=
a_i(u,\varphi)
-
\frac{1}{\tau\ubar{\omega}_i}
B_i(u)\scal{p_i}{\varphi},
\qquad u,\varphi \in \hko{1}, 
\end{equation}
where $a_i$ is defined in \eqref{eq:bil_form_a}. Note that $a_i$ is continuous and coercive on $\hko{1}$, see \Cref{thm:dp:discrete}. We denote the associated operator by $A_i:\hko{1}\to {\hko{1}}^\ast$. 
Consequently, problem \eqref{equiv:var_for_inv_disc_prob2} is  equivalent to
the operator equation
\begin{equation}\label{eq:abstract_rank_one}
\text{Find } u_i \in \hko{1} \text{ such that }
\qquad
\bigl(A_i - \alpha_i p_i \otimes B_i\bigr) u_i = g_i,
\end{equation}
where
\[
\alpha_i := \frac{1}{\tau\ubar{\omega}_i},
\qquad
(p_i \otimes B_i)(u) := B_i(u)p_i.
\]
Note that the unperturbed problem
\begin{equation}\label{eq:u0_i_def}
a_i(u_i^{(0)},\varphi) = g_i(\varphi),
\qquad \forall \varphi \in \hko{1},
\end{equation}
admits a unique solution $u_i^{(0)}\in \hko{1}$.
Introducing the auxiliary function $z_i\in \hko{1}$ as the unique solution of
\begin{equation}\label{eq:z_i_def}
a_i(z_i,\varphi) = \scal{p_i}{\varphi},
\qquad \forall \varphi \in \hko{1},
\end{equation}
we can search the solution $u_i$ of the  perturbed problem
\eqref{eq:abstract_rank_one} in the form
\begin{equation}\label{eq:ansatz_i}
u_i = u_i^{(0)} + c_i z_i,
\end{equation}
where $c_i\in\RR$ is a scalar coefficient to be determined.
Substituting \eqref{eq:ansatz_i} into \eqref{eq:abstract_rank_one} yields
\[
A_i(u_i^{(0)} + c_i z_i)
-
\alpha_i B_i(u_i^{(0)} + c_i z_i) p_i
=
g_i.
\]
Using the defining relations \eqref{eq:u0_i_def} and \eqref{eq:z_i_def},
this expression reduces to
\[
g_i
+ c_i p_i
-
\alpha_i\bigl(B_i(u_i^{(0)}) + c_i B_i(z_i)\bigr)p_i
=
g_i.
\]
and thus
\[
\bigl[
c_i
-
\alpha_i B_i(u_i^{(0)})
-
\alpha_i c_i B_i(z_i)
\bigr] p_i
= 0.
\]
Since $p_i\neq 0$, the scalar coefficient must vanish, which leads to
\begin{equation}\label{eq:ci_formula}
c_i
=
\frac{\alpha_i B_i(u_i^{(0)})}
{1 - \alpha_i B_i(z_i)},
\end{equation}
provided that 
\begin{equation}\label{eq:nondegeneracy}
1 - \alpha_i B_i(z_i) \neq 0.
\end{equation}
Combining \eqref{eq:ansatz_i} and \eqref{eq:ci_formula}, the solution of
\eqref{equiv:var_for_inv_disc_prob} admits the representation
\begin{equation}\label{eq:solution_representation}
u_i
=
u_i^{(0)}
+
\frac{\alpha_i B_i(u_i^{(0)})}
{1 - \alpha_i B_i(z_i)}
z_i,
\end{equation}
where $u_i^{(0)}$ and $z_i$ are defined by
\eqref{eq:u0_i_def}–\eqref{eq:z_i_def}. Afterwards, we can derive $h_i$ from \eqref{disc:hi-1}. This strategy is summarised in \Cref{algorithm}. In the following remark, we show that condition \eqref{eq:nondegeneracy} is satisfied.

\begin{remark}
Condition \eqref{eq:nondegeneracy} is ensured by the smallness assumption \ref{as:uniqueness} for any $i=1,\ldots,n$.
By definition, we have
\[
\alpha_i B_i(z_i)
=
\frac{1}{\tau\ubar{\omega}_i}
\scal{\eta_i\nabla z_i}{\nabla\!\left(\frac{\omega}{\rho_i}\right)}.
\]
Choosing $\varphi=z_i$ in \eqref{eq:z_i_def} yields (see \Cref{thm:dp:discrete}) that
\[
\frac{\ubar{\rho}_0}{\tau} \vnorma{z_i}^2  +\frac{\ubar{\eta}_0}{\tau}\vnorma{\nabla z_i}^2
\le \scal{p_i}{z_i}. 
\]
For the right-hand side, the Cauchy-Schwarz
and Young inequalities give
\[
(p_i,z_i)
\le
\frac{1}{2\ubar{\rho}_0}\|p_i\|^2
+
\frac{\ubar{\rho}_0}{2}\|z_i\|^2.
\]
Hence, as $\tau < 1$, we have
\[
\frac{\ubar{\eta}_0}{\tau}\vnorma{\nabla z_i}^2 \le \frac{1}{2\ubar{\rho}_0}\|p_i\|^2 
\]
and thus
\begin{equation*}\label{eq:gradzi_bound}
	\|\nabla z_i\|
	\le
	\sqrt{\frac{\tau}{2\ubar{\rho}_0\ubar{\eta}_0} }\|p_i\| \le  \frac{\tau}{2 \sqrt{\ubar{\rho}_0\ubar{\eta}_0}}  \|p\|_{\mathcal X}.
\end{equation*}
Therefore,  using the Cauchy-Schwarz inequality, \ref{as:p} and \ref{as:uniqueness}, we obtain
\[
\abs{\alpha_i B_i(z_i)} \le \frac{1}{\tau\ubar{\omega}_m} 
\ubar{\eta}_1
\|\nabla z_i\|
\left\|\nabla\!\left(\frac{\omega}{\rho_i}\right)\right\| \le 
\frac{\ubar{\eta}_1 M \|p\|_{\mathcal X}}{2\sqrt{\ubar{\rho}_0\ubar{\eta}_0} \ubar{\omega}_m} < \sqrt{\zeta} < 1.
\]
\end{remark}

\begin{algorithm}[htbp]
\caption{Time-stepping algorithm for the discrete inverse problem}
\label{algorithm}
\begin{algorithmic}[1]
	\Require initial state $u_0=\tilde u_0$, time step $\tau$, final time $T$, measurement $m(t)$
	\Ensure discrete solution $\{u_i\}_{i=0}^{n}$ and reconstructed source $\{h_i\}_{i=1}^{n}$
	
	\For{$i = 1$ to $n$}
	\State {\bf Compute } $m_i'$:
	\[
	m_i^\prime =
	\begin{cases}
		m^\prime(t_i), & \epsilon = 0,\\
		{(p^{\epsilon}_{\tilde{d}})}^{\prime}(t_i), & \epsilon \neq 0.
	\end{cases}
	\]
	
	\State {\bf Assemble } the coercive bilinear form
	\[
	a_i(u,\varphi)
	=
	\frac{\rho}{\tau}(u,\varphi)
	+
	\frac{\eta}{\tau}(\nabla u,\nabla\varphi)
	+
	(\kappa_i\nabla u,\nabla\varphi).
	\]
	
	\State {\bf Assemble } the linear functional 
	$g_i(\cdot)$, using $u_{i-1}$ and $m_i^\prime$.
	
	\State {\bf Solve } the unperturbed problem:
	\[
	a_i(u_i^{(0)},\varphi) = g_i(\varphi),
	\qquad \forall \varphi \in \hko{1}.
	\]
	
	\State {\bf Solve } the auxiliary problem:
	\[
	a_i(z_i,\varphi) = (p_i,\varphi),
	\qquad \forall \varphi \in \hko{1}.
	\]
	
	\State {\bf Compute } the scalar quantities
	\[
	B_i(u_i^{(0)}) = \scal{\eta_i\nabla u_i^{(0)}}{\nabla(\omega/\rho_i)},
	\qquad
	B_i(z_i) = \scal{\eta_i\nabla z_i}{\nabla(\omega/\rho_i)}.
	\]
	
	\State {\bf Compute } the coefficient
	\[
	c_i = \frac{\alpha_i B_i(u_i^{(0)})}{1-\alpha_i B_i(z_i)},
	\qquad
	\alpha_i = \frac{1}{\tau\bar\omega_i}.
	\]
	
	\State {\bf Recover } the solution
	\[
	u_i = u_i^{(0)} + c_i z_i.
	\]
	
	\State {\bf Compute } $h_i$ explicitly from
	\eqref{disc:hi-1} using $u_i$ and $u_{i-1}$.
	
	\EndFor
\end{algorithmic}
\end{algorithm}

\subsection{Convergence rate for exact data}
\label{subsec:noisy_free}

In this subsection, we investigate the convergence rate for exact data. We first show the convergence result. 

\begin{theorem}[Discrete-time error estimate]\label{thm:IP_disc_error}
Let Assumptions~\ref{as:DP:rho}-\ref{as:uniqueness} be satisfied.
Assume additionally that
\[
\rho \in \ckIX{2}{\lp{\infty}},
\qquad
m\in \Cont^2([0,T]),
\qquad
u\in \ckIX{2}{\hko{1}}.
\]
Let $\tau=T/n$, $t_i=i\tau$, and let $(u_i,h_i)_{i=1}^n$ be the time-discrete solutions
generated by \eqref{eq:disc_inv_prob}-\eqref{disc:hi-1} (with exact data, i.e. 
$m_i^\prime=m^\prime(t_i)$.).
Then there exist constants $C_u>0$ and $\tau_0>0$, independent of $\tau$, such that
\[
E_{\max}^u(\tau):=\max_{1\le i\le n}\vnorma{u(t_i)-u_i}_{\hk{1}}
\le C_u\tau,
\]
for all $\tau \le \tau_0.$
Moreover, then there exists $C_h>0$, independent of $\tau$, such that
\[
E_{\ell^2}^h(\tau):= \left(\sum_{i=1}^n \abs{h(t_i)-h_i}^2 \tau\right)^{\half} 
\le C_h\tau.
\]
\end{theorem}

\begin{proof}
We split the proof in multiple steps to increase the readability. \\
{Step 1: Notations and residuals.}
For $i=0,\dots,n$, we set
\[
e_{u_i}:=u(t_i)-u_i\in \hko{1},
\qquad
e_{h_i}:=h(t_i)-h_i\in\RR,
\qquad  
\]
Note that $e_{u_0}=0.$
We also remind the notations
\[
\delta v_i:=\frac{v_i-v_{i-1}}{\tau},
\quad
\delta(\rho_i v_i):=\frac{\rho_i v_i-\rho_{i-1}v_{i-1}}{\tau},
\quad
\rho_i:=\rho(t_i),\ \eta_i:=\eta(t_i),\ \kappa_i:=\kappa(t_i),\ \ubar{\omega}_i:=\ubar{\omega}(t_i).
\]
For the continuous solution, we define
\[
\delta u(t_i):=\frac{u(t_i)-u(t_{i-1})}{\tau}.
\]
Since $u\in \ckIX{2}{\hko{1}}$, we define the (Backward Euler) consistency residual
\begin{equation}\label{eq:ri_def_new}
	r_i:=\delta u(t_i)-\pdt u(t_i)\in \hko{1}.
\end{equation}
By the
Taylor identity, we get
\begin{equation}\label{eq:BE_consistency_u}
	\frac{u(t_i)-u(t_{i-1})}{\tau}
	=\partial_t u(t_i) - \frac{\tau}{2}\partial_{tt}u(\xi_i),
	\qquad \text{ for some } \xi_i\in(t_{i-1},t_i),
\end{equation}
hence
\begin{equation}\label{eq:BE_consistency_u_bound}
	\vnorma{r_i}_{\hk{1}}
	\le C\tau,
\end{equation}
where $C$ only depends  on $\vnorma{\pdtt u}_{\lpkIX{\infty}{\hk{1}}}$.
Likewise, since $\rho\in \ckIX{2}{\lp{\infty}}$ and $u\in \ckIX{2}{\lp{2}}$,
we can define the product residual
\begin{equation}\label{eq:rhor_def_new}
	r_i^\rho:=\delta(\rho(t_i)u(t_i))-\pdt(\rho u)(t_i)\in \lp{2},
	\qquad
	\vnorma{r_i^\rho}\le C\tau,
\end{equation}
with $C$ depending on $\vnorma{\rho}_{\ckIX{2}{\lp{\infty}}}$ and $\vnorma{u}_{\ckIX{2}{\lp{2}}}.$\\
{Step 2: Weak formulation for the exact solution at time $t_i$.}
Evaluating the continuous variational formulation \eqref{eq:var_for} at time $t=t_i$ gives
for all $\varphi\in \hko{1}$ that
\[
\scal{\pdt(\rho u)(t_i)}{\varphi}
+\scal{\eta_i\nabla\pdt u(t_i)}{\nabla\varphi}
+\scal{\kappa_i\nabla u(t_i)}{\nabla\varphi}
=
h(t_i)\scal{p_i}{\varphi}+\scal{f(u(t_i))}{\varphi}.
\]
Using \eqref{eq:ri_def_new}-\eqref{eq:rhor_def_new},
we obtain the time-discrete identity
\begin{equation}\label{eq:cont_disc_identity_new}
	\scal{\delta(\rho(t_i)u(t_i))}{\varphi}
	+\scal{\eta_i\nabla\delta u(t_i)}{\nabla\varphi}
	+\scal{\kappa_i\nabla u(t_i)}{\nabla\varphi}
	=
	h(t_i)\scal{p_i}{\varphi}+\scal{f(u(t_i))}{\varphi}
	+\mathcal{R}_i(\varphi),
\end{equation}
where the residual functional $\mathcal{R}_i$ is defined by
\begin{equation}\label{eq:Resi_def_new}
	\mathcal{R}_i(\varphi):=\scal{r_i^\rho}{\varphi}+\scal{\eta_i\nabla r_i}{\nabla\varphi}.
\end{equation}
By the Cauchy-Schwarz inequality and \eqref{eq:BE_consistency_u_bound}-\eqref{eq:rhor_def_new}, we have
\[
\abs{\mathcal{R}_i(\varphi)}\le C\tau\vnorma{\varphi}_{\hk{1}}.
\]
Applying the Friedrichs' inequality implies 
\begin{equation}\label{eq:Resi_bound_new}
	\abs{\mathcal{R}_i(\varphi)}\le C\tau\vnorma{\nabla \varphi}.
\end{equation}

{Step 3: Discrete error equation.}
Now, we subtract the discrete scheme \eqref{eq:disc_inv_prob} from \eqref{eq:cont_disc_identity_new} to  obtain for all $\varphi\in\hko{1}$ that
\begin{equation}\label{eq:error_eq_new}
	\scal{\delta(\rho_i e_{u_i})}{\varphi}
	+\scal{\eta_i\nabla\delta e_{u_i}}{\nabla\varphi}
	+\scal{\kappa_i\nabla e_{u_i}}{\nabla\varphi}
	=
	e_{h_i}\scal{p_i}{\varphi}
	+\scal{f(u(t_i))-f(u_{i-1})}{\varphi}
	+\mathcal{R}_i(\varphi).
\end{equation}

{Step 4: Derivation estimate $e_{h_i}$}
We now derive an explicit formula for $e_{h_i}=h(t_i)-h_i$ by comparing
the continuous identity \eqref{eq:expression_h} with the discrete definition
\eqref{disc:hi-1}. Evaluating \eqref{eq:expression_h} at $t=t_i$ gives
\begin{multline}\label{eq:h_cont_ti_full}
	h(t_i)
	=
	\frac{1}{\ubar{\omega}_i}
	\Bigg[
	m^\prime(t_i)
	+
	\scal{u(t_i)}{\frac{\omega(\partial_t\rho)(t_i)}{\rho_i}}
	+
	\scal{\eta_i\nabla\partial_t u(t_i)}{\nabla\!\Big(\frac{\omega}{\rho_i}\Big)}\\
	+
	\scal{\kappa_i\nabla u(t_i)}{\nabla\!\Big(\frac{\omega}{\rho_i}\Big)}
	-
	\scal{f(u(t_i))}{\frac{\omega}{\rho_i}}
	\Bigg].
\end{multline}
Subtracting \eqref{disc:hi-1} from \eqref{eq:h_cont_ti_full} gives
\begin{equation}\label{eq:eh_expanded}
	e_{h_i}
	=
	\frac{1}{\ubar{\omega}_i}
	\Big(
	\mathcal{H}_{1,i}+\mathcal{H}_{2,i}+\mathcal{H}_{3,i}+\mathcal{H}_{4,i}
	\Big),
\end{equation}
where
\begin{align*}
	\mathcal{H}_{1,i}
	&:=
	\scal{u(t_i)}{\frac{\omega(\partial_t\rho)(t_i)}{\rho_i}}
	-
	\scal{u_{i-1}}{\frac{\omega\delta\rho_i}{\rho_i}},
	\\
	\mathcal{H}_{2,i}
	&:=
	\scal{\eta_i\nabla\partial_t u(t_i)}{\nabla\!\Big(\frac{\omega}{\rho_i}\Big)}
	-
	\scal{\eta_i\nabla\delta u_i}{\nabla\!\Big(\frac{\omega}{\rho_i}\Big)},
	\\
	\mathcal{H}_{3,i}
	&:=
	\scal{\kappa_i\nabla u(t_i)}{\nabla\!\Big(\frac{\omega}{\rho_i}\Big)}
	-
	\scal{\kappa_i\nabla u_{i-1}}{\nabla\!\Big(\frac{\omega}{\rho_i}\Big)},
	\\
	\mathcal{H}_{4,i}
	&:=
	-\scal{f(u(t_i))}{\frac{\omega}{\rho_i}}
	+\scal{f(u_{i-1})}{\frac{\omega}{\rho_i}}.
\end{align*}
We now bound each $\mathcal{H}_{k,i}$ explicitly. In the term $\mathcal{H}_{1,i},$
we add and subtract $u(t_{i-1})$ and $(\partial_t\rho)(t_i)$ to obtain
\begin{multline*}
	\mathcal{H}_{1,i}
	=
	\underbrace{\scal{u(t_i)-u(t_{i-1})}{\frac{\omega(\partial_t\rho)(t_i)}{\rho_i}}}_{\mathcal{H}_{1,1,i}}
	+
	\underbrace{\scal{u(t_{i-1})-u_{i-1}}{\frac{\omega(\partial_t\rho)(t_i)}{\rho_i}}}_{\mathcal{H}_{1,2,i}} \\
	+
	\underbrace{\scal{u_{i-1}}{\frac{\omega\big((\partial_t\rho)(t_i)-\delta\rho_i\big)}{\rho_i}}}_{\mathcal{H}_{1,3,i}}.
\end{multline*}
Note that since $\rho\in \ckIX{2}{\lp{\infty}}$, it holds that 
\begin{equation*}\label{eq:BE_consistency_rho}
	\delta\rho_i
	=\partial_t\rho(t_i) - \frac{\tau}{2}\partial_{tt}\rho(\theta_i),
	\qquad \theta_i\in(t_{i-1},t_i),
\end{equation*}
and therefore
\begin{equation}\label{eq:BE_consistency_rho_bound}
	\|\partial_t\rho(t_i)-\delta\rho_i\|_{\lp{\infty}}
	\le \frac{\tau}{2}\|\partial_{tt}\rho\|_{\lpkIX{\infty}{\lp{\infty}}}.
\end{equation}
Moreover, since $u\in \ckIX{1}{\hk{1}}$, we have \begin{equation}\label{eq:convergence_u}
	\vnorma{u(t_i)-u(t_{i-1})}_{\hk{1}}\le C\tau.
\end{equation}
Then, by \eqref{eq:convergence_u} and \eqref{eq:BE_consistency_rho_bound}, we get
\[
|\mathcal{H}_{1,1,i}|
\le
\|u(t_i)-u(t_{i-1})\|
\left\|\frac{\omega(\partial_t\rho)(t_i)}{\rho_i}\right\|_{\lp{\infty}}
\le C\tau,
\]
\[
|\mathcal{H}_{1,2,i}|
\le
\|e_{u_{i-1}}\|
\left\|\frac{\omega(\partial_t\rho)(t_i)}{\rho_i}\right\|_{\lp{\infty}}
\le C\|e_{u_{i-1}}\|,
\]
\[
|\mathcal{H}_{1,3,i}|
\le
\|u_{i-1}\|
\left\|\frac{\omega\big((\partial_t\rho)(t_i)-\delta\rho_i\big)}{\rho_i}\right\|_{\lp{\infty}}
\le C\tau.
\]
Hence, we have 
\begin{equation}\label{eq:H1_bound}
	|\mathcal{H}_{1,i}|
	\le
	C\tau + C\|e_{u_{i-1}}\|.
\end{equation}
We rewrite the term $\mathcal{H}_{2,i}$ as
\[
\mathcal{H}_{2,i}
=
-\scal{\eta_i\nabla r_i}{\nabla(\omega/\rho_i)}
+
\scal{\eta_i\nabla\delta e_{u_i}}{\nabla(\omega/\rho_i)}.
\]
Using \eqref{eq:BE_consistency_u_bound} and Cauchy-Schwarz, we obtain
\[
\abs{\scal{\eta_i\nabla r_i}{\nabla(\omega/\rho_i)}}
\le
\|\eta_i\|_{\lp{\infty}}\|r_i\|_{\hk{1}}\Big\|\nabla\Big(\frac{\omega}{\rho_i}\Big)\Big\|
\le C\tau,
\]
and
\[
\abs{\scal{\eta_i\nabla\delta e_{u_i}}{\nabla(\omega/\rho_i)}}
\le
\|\eta_i\|_{\lp{\infty}}\|\nabla \delta e_{u_i}\|\Big\|\nabla\Big(\frac{\omega}{\rho_i}\Big)\Big\|
\le
\ubar\eta_1 M\|\nabla \delta e_{u_i}\|.
\]
Therefore, we get
\begin{equation}\label{eq:H2_bound}
	|\mathcal{H}_{2,i}|
	\le
	C\tau + \ubar\eta_1 M\|\nabla \delta e_{u_i}\|.
\end{equation}
For the term $\mathcal{H}_{3,i}$, we write
\begin{equation} \label{eq:trick}
	u(t_i)-u_{i-1}
	=
	\big(u(t_i)-u(t_{i-1})\big)+e_{u_{i-1}},
\end{equation}
hence
\[
|\mathcal{H}_{3,i}|
\le
\ubar\kappa_1\|\nabla(u(t_i)-u(t_{i-1}))\|\Big\|\nabla\Big(\frac{\omega}{\rho_i}\Big)\Big\|
+
\ubar\kappa_1\|\nabla e_{u_{i-1}}\|\Big\|\nabla\Big(\frac{\omega}{\rho_i}\Big)\Big\|.
\]
Using \eqref{eq:convergence_u}, we obtain
\begin{equation}\label{eq:H3_bound}
	|\mathcal{H}_{3,i}|
	\le
	C\tau + C\|\nabla e_{u_{i-1}}\|.
\end{equation}
Finally, by the Lipschitz continuity of $f$, \eqref{eq:trick} and \eqref{eq:convergence_u}, we get 
\begin{equation}\label{eq:H4_bound}
	\abs{\mathcal{H}_{4,i}}
	\le
	\left\|\frac{\omega}{\rho_i}\right\|
	\|f(u(t_i))-f(u_{i-1})\|
	\le
	C\|u(t_i)-u_{i-1}\|
	\le
	C\tau + C\|e_{u_{i-1}}\|.
\end{equation}
Combining \eqref{eq:eh_expanded} with \eqref{eq:H1_bound}, \eqref{eq:H2_bound},
\eqref{eq:H3_bound}, \eqref{eq:H4_bound}, and using $|\ubar{\omega}_i|\ge \ubar\omega_m$,
we arrive at the  bound
\begin{equation*}
	\abs{e_{h_i}}
	\le
	C\tau
	+
	C\|e_{u_{i-1}}\|_{\hk{1}}
	+
	\frac{\ubar\eta_1 M}{\ubar\omega_m} \|\nabla \delta e_{u_i}\|. 
\end{equation*}
By the Friedrichs' inequality, we can write 
\begin{equation}\label{eq:eh_final_bound}
	\abs{e_{h_i}}
	\le
	C\tau
	+
	C\|\nabla e_{u_{i-1}}\|
	+
	\frac{\ubar\eta_1 M}{\ubar\omega_m} \|\nabla \delta e_{u_i}\|,
	\qquad i=1,\dots,n.
\end{equation}

{Step 5: Derivation error bound $e_{u_i}$.}
We take $\varphi=\delta e_{u_i}\tau$ in \eqref{eq:error_eq_new} and sum up the result for $i=1,\ldots,j$ with $1\le j\le n$. 
We obtain
\begin{multline} \label{eq:error_eq_new2}
	\sum_{i=1}^j  \scal{\delta(\rho_i e_{u_i})}{\delta e_{u_i}} \tau
	+ \sum_{i=1}^j \scal{\eta_i\nabla\delta e_{u_i}}{\nabla \delta e_{u_i}} \tau
	+ \sum_{i=1}^j \scal{\kappa_i\nabla e_{u_i}}{\nabla \delta e_{u_i}} \tau 
	=\\
	\sum_{i=1}^j e_{h_i}\scal{p_i}{\delta e_{u_i}} \tau
	+ \sum_{i=1}^j\scal{f(u(t_i))-f(u_{i-1})}{\delta e_{u_i}} \tau
	+ \sum_{i=1}^j\mathcal{R}_i(\delta e_{u_i})\tau.
\end{multline}
We treat each term on the left-hand side as in \Cref{direct_problem:a_priori_estimate}, replacing $u_i$ by $e_{u_i}$ and using $e_{u_0}=0$.  We obtain
\begin{equation} \label{eq:rho_lower_new}
	\sum_{i=1}^j \scal{\delta (\rho_i e_{u_i})}{\delta e_{u_i}}\tau 
	\ge \left(\ubar{\rho}_0 - {\veps_1}  \right) \sum_{i=1}^j \vnorma{\delta e_{u_i}}^2 \tau - C_{\veps_1} \sum_{i=1}^j \vnorma{\nabla e_{u_i}}^2 \tau, 
\end{equation}
\begin{equation} \label{eq:eta_term_lower}
	\sum_{i=1}^j \scal{\eta_i\nabla\delta e_{u_i}}{\nabla \delta e_{u_i}} \tau \ge  \ubar{\eta}_0  \sum_{i=1}^j \vnorma{\nabla\delta e_{u_i}}^2 \tau,
\end{equation}
and 
\begin{equation} \label{eq:kappa_lower}
	\sum_{i=1}^j  \scal{\kappa_i \nabla e_{u_i}}{\nabla \delta e_{u_i}} \tau
	\geq \frac{\ubar{\kappa}_0}{2}  \vnorma{\nabla e_{u_j}}^2 
	- \frac{\ubar{\kappa}_1^\prime}{2}  \sum_{i=1}^{j-1} \vnorma{\nabla e_{u_i}}^2 \tau  
	+ \frac{\ubar{\kappa}_0}{2} \sum_{i=1}^j \vnorma{\nabla e_{u_i} - \nabla e_{u_{i-1}}}^2.  
\end{equation}
Now, we handle the right-hand side of \eqref{eq:error_eq_new2}. 
It follows from the Lipschitz continuity of $f$ that
\[
\vnorma{f(u(t_i))-f(u_{i-1})}
\le
L_f\vnorma{u(t_i)-u_{i-1}}
\le
L_f\Big(\vnorma{u(t_i)-u(t_{i-1})}+\vnorma{e_{u_{i-1}}}\Big).
\]
Hence, from \eqref{eq:convergence_u}, the Friedrichs' inequality and $e_{u_0}=0$ it follows that 
\begin{align}
	\sum_{i=1}^j \abs{\scal{f(u(t_i))-f(u_{i-1})}{\delta e_{u_i}}} \tau
	&\le
	C \sum_{i=1}^j \Big(\tau+\vnorma{e_{u_{i-1}}}\Big)\vnorma{\delta e_{u_i}} \tau
	\nonumber\\
	&\le
	C_{\veps_2}\tau^2 + C_{\veps_2}  \sum_{i=1}^{j-1} \vnorma{\nabla e_{u_{i}}}^2 \tau + \veps_2 \sum_{i=1}^j \vnorma{\delta e_{u_i}}^2 \tau.
	\label{eq:nonlinear_rhs}
\end{align}
Moreover, using \eqref{eq:eh_final_bound}, we get 
\[
\abs{e_{h_i}\scal{p_i}{\delta e_{u_i}}} \le \vnorma{p}_{\mathcal{X}} \Big(C\tau
+
C\|\nabla e_{u_{i-1}}\|
+
\frac{\ubar\eta_1 M}{\ubar\omega_m}  \|\nabla \delta e_{u_i}\|\Big)\|\delta e_{u_i}\|. 
\]
We handle all terms on the right-hand side by using the $\veps$-Young inequality. 
We obtain 
\[
\vnorma{p}_{\mathcal{X}} C \tau \|\delta e_{u_i}\| \le C_{\veps_3}\tau^2 + \veps_3 \|\delta e_{u_i}\|^2,
\]
\[
\vnorma{p}_{\mathcal{X}} C\|\nabla e_{u_{i-1}}\| \|\delta e_{u_i}\| 
\le C_{\veps_4} \|\nabla e_{u_{i-1}}\|^2  + \veps_4 \|\delta e_{u_i}\|^2,
\]
and
\[
\vnorma{p}_{\mathcal{X}} \frac{\ubar\eta_1 M}{\ubar\omega_m} \|\nabla \delta e_{u_i} \| \|\delta e_{u_i}\|  \le \frac{\vnorma{p}_{\mathcal{X}}^2\ubar{\eta}_1^2 M^2}{\ubar{\omega}^2_m  4 \veps_5} \| \nabla  \delta e_{u_i} \|^2 + \veps_5 \|\delta e_{u_i}\|^2.
\]
Therefore, we obtain (as $e_{u_0}=0$) 
\begin{multline} \label{eq:est_e_h_term}
	\sum_{i=1}^j \abs{e_{h_i}\scal{p_i}{\delta e_{u_i}}} \tau
	\le
	C_{\veps_3}\tau^2  + (\veps_3+\veps_4 + \veps_5) \sum_{i=1}^j  \|\delta e_{u_i}\|^2 \tau \\
	+ C_{\veps_4} \sum_{i=1}^{j-1} \|\nabla e_{u_{i}}\|^2 \tau + \frac{\vnorma{p}_{\mathcal{X}}^2\ubar{\eta}_1^2 M^2}{\ubar{\omega}^2_m  4 \veps_5} \sum_{i=1}^j \| \nabla  \delta e_{u_i} \|^2.
\end{multline}
Additionally, by \eqref{eq:Resi_bound_new}, we get that
\begin{equation}\label{eq:residual_rhs}
	\sum_{i=1}^j \abs{\mathcal{R}_i(\delta e_{u_i})} \tau
	\le C\tau \sum_{i=1}^j \vnorma{\nabla \delta e_{u_i}}\tau
	\le C_{\veps_6}\tau^2 + \veps_6 \sum_{i=1}^j\vnorma{\nabla \delta e_{u_i}}^2 \tau.
\end{equation}
Now, putting $\veps_5 = \ubar{\rho}_0\zeta $ and collecting all estimates \eqref{eq:rho_lower_new}, \eqref{eq:eta_term_lower},
\eqref{eq:kappa_lower}, \eqref{eq:residual_rhs}, \eqref{eq:nonlinear_rhs}, and \eqref{eq:est_e_h_term}, we arrive at
\begin{multline}
	\left(\ubar{\rho}_0(1-\zeta) -  \sum_{i=1}^4 \veps_i  \right) \sum_{i=1}^j \vnorma{\delta e_{u_i}}^2 \tau + \left(  \ubar{\eta}_0 - \veps_6 - \frac{\vnorma{p}_{\mathcal{X}}^2\ubar{\eta}_1^2 M^2}{\ubar{\omega}^2_m  4 \ubar{\rho}_0\zeta} \right) \sum_{i=1}^j \vnorma{\nabla\delta e_{u_i}}^2 \tau \\
	+  \frac{\ubar{\kappa}_0}{2}  \vnorma{\nabla e_{u_j}}^2  
	+ \frac{\ubar{\kappa}_0}{2} \sum_{i=1}^j \vnorma{\nabla e_{u_i} - \nabla e_{u_{i-1}}}^2 \\
	\le \left(C_{\veps_2} + C_{\veps_3} + C_{\veps_6}\right)\tau^2  + \left( \frac{\ubar{\kappa}_1^\prime}{2} + C_{\veps_1}  + C_{\veps_2}  + C_{\veps_4} \right) \sum_{i=1}^{j} \vnorma{\nabla e_{u_i}}^2 \tau.
\end{multline}
Now, we choose $\veps_6
\le
\frac12\left(
\ubar{\eta}_0
-
\frac{\|p\|_{\mathcal X}^2 \ubar{\eta}_1^2 M^2}
{\ubar{\omega}_m^24\ubar{\rho}_0\zeta}
\right)
$. This is possible due to the smallness condition 
\eqref{cond:uniqueness} in \ref{as:uniqueness}. Moreover, we take $\veps_i = \frac{\ubar{\rho}_0(1-\zeta)}{8}$ for $i=1,\ldots,4.$ Therefore, we obtain by the Gr\"onwall lemma the existence of constants  $C>0$ (independent of $\tau$) and $\tau_0>0$ such that 
\begin{equation}\label{eq:error_estimate_eu}
	\sum_{i=1}^j \vnorma{\delta e_{u_i}}_{\hk{1}}^2 \tau  + \vnorma{e_j}_{\hk{1}}^2 + \sum_{i=1}^j \vnorma{\nabla e_{u_i} - \nabla e_{u_{i-1}}}_{\hk{1}}^2 \le C\tau^2,
\end{equation}
for all $\tau \le \tau_0.$

{Step 6: Derivation error bound $e_{h_i}$.} From \eqref{eq:eh_final_bound}, we have that
\begin{equation*}\label{eq:eh_final_bound2}
	\sum_{i=1}^j \abs{e_{h_i}}^2\tau
	\le
	C\tau^2
	+
	C \sum_{i=1}^{j-1}\|\nabla e_{u_{i}}\|^2 \tau
	+
	C \sum_{i=1}^{j} \|\nabla \delta e_{u_i}\|^2 \tau. 
\end{equation*}
From the estimate \eqref{eq:error_estimate_eu}, we finally have the existence of constants  $C>0$ and $\tau_0$ such that 
\[
\sum_{i=1}^j \abs{e_{h_i}}^2\tau \le C \tau^2, 
\]
for all $\tau \le \tau_0.$
\end{proof}

We next illustrate the theoretical convergence rates by applying \Cref{algorithm} to
Experiment~2 (corresponding to the exact solution \eqref{eq:exact_solution2}).
The time step is chosen as $\tau = 2^{-i}$ for $i = 7,\ldots,15$.
The numerical results exhibit the expected first-order convergence obtained in \Cref{thm:IP_disc_error},
namely
\[
E_{\mathrm{max}}^{u_2}(\tau) = \mathcal{O}(\tau),
\qquad
E_{\ell^2}^{h_2}(\tau) = \mathcal{O}(\tau),
\]
as shown in \Cref{Experiment1conv}.
Moreover, for $\tau = 2^{-13}$ the reconstruction error satisfies
$E_{\ell^2}^{h_2} \approx 10^{-2}$, which is sufficiently small to justify
the use of this time step in the subsequent experiments with noisy data.

\begin{figure}[htbp]
\begin{center}
	\subfigure[]{\includegraphics[width=0.75\textwidth,angle=0,height = 0.22\textheight]{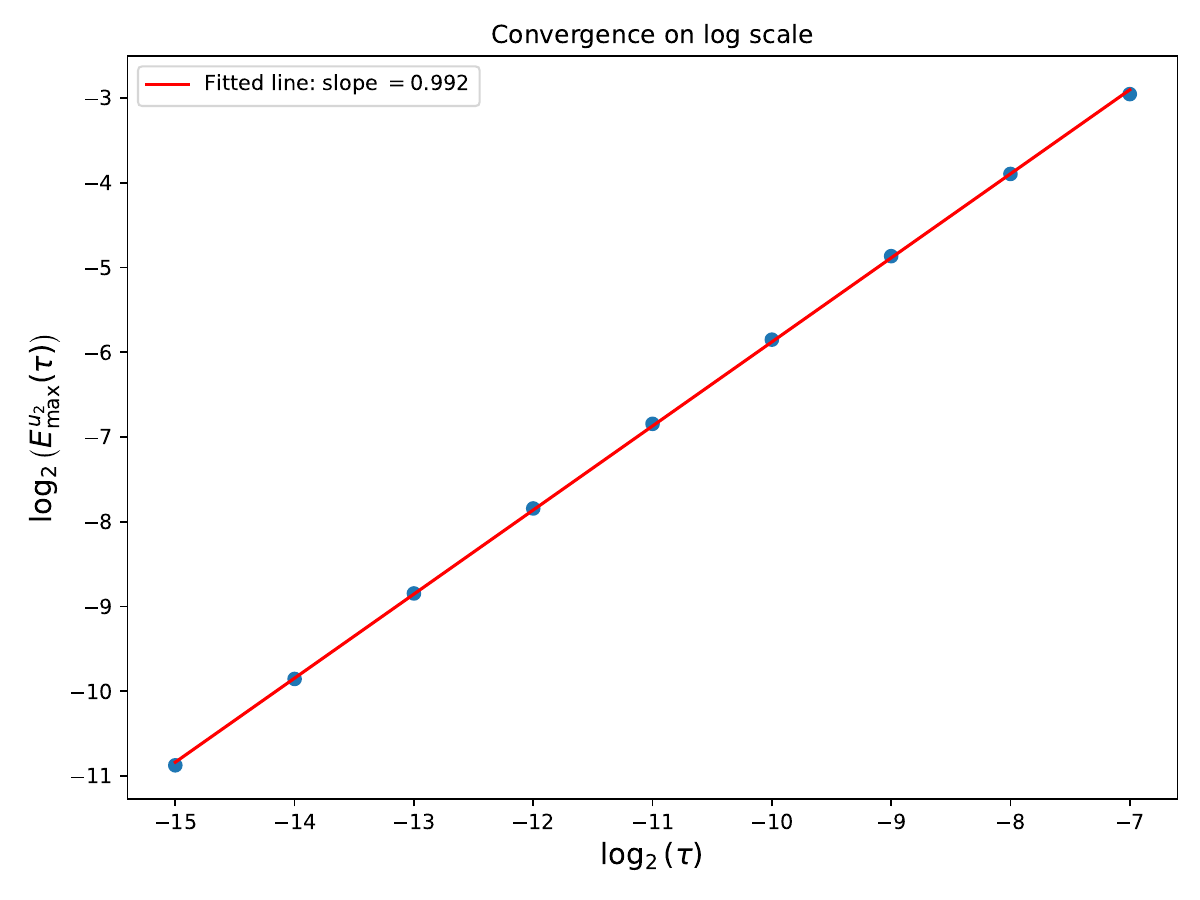}}
	\subfigure[]{\includegraphics[width=0.75\textwidth,angle=0,height = 0.22\textheight]{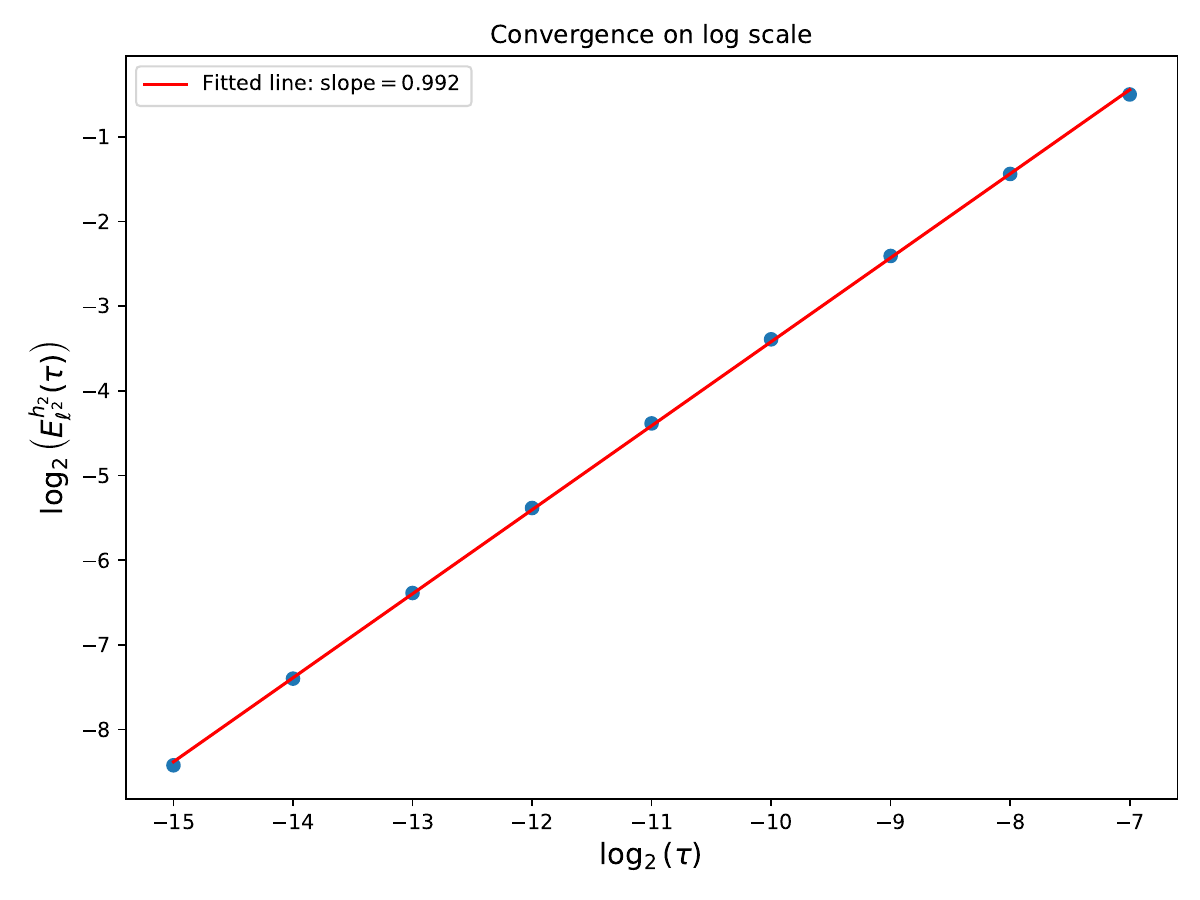}}
\end{center}
\vspace{-0.5cm}
\caption{Experiment 2: (a) rate of convergence of $u_2$; (b) rate of convergence for $h_2$, for noise-free data $m_2$. 
}
\label{Experiment1conv}
\end{figure}

\subsection{Noisy experiments in one dimension}
\label{subsec:1dnoisy}

The results of the first experiment are shown in \Cref{Experiment1,Experiment1u}, and those of Experiment~2 in \Cref{Experiment2,Experiment2u}. 
For both one-dimensional experiments, the figures illustrate that accurate approximations of the exact sources and the solution at the final time are obtained across all considered noise levels 
$\epsilon = \{0.001,0.005,0.01,0.03,0.05\}$ for $\tau = 2^{-13}$. The reconstructions are stable for small noise levels, while the error naturally increases as the noise level increases, highlighting the impact of measurement noise on the inverse problem.

\begin{figure}[htbp]
\begin{center}
	\subfigure[]{\includegraphics[width=0.75\textwidth,angle=0,height = 0.22\textheight]{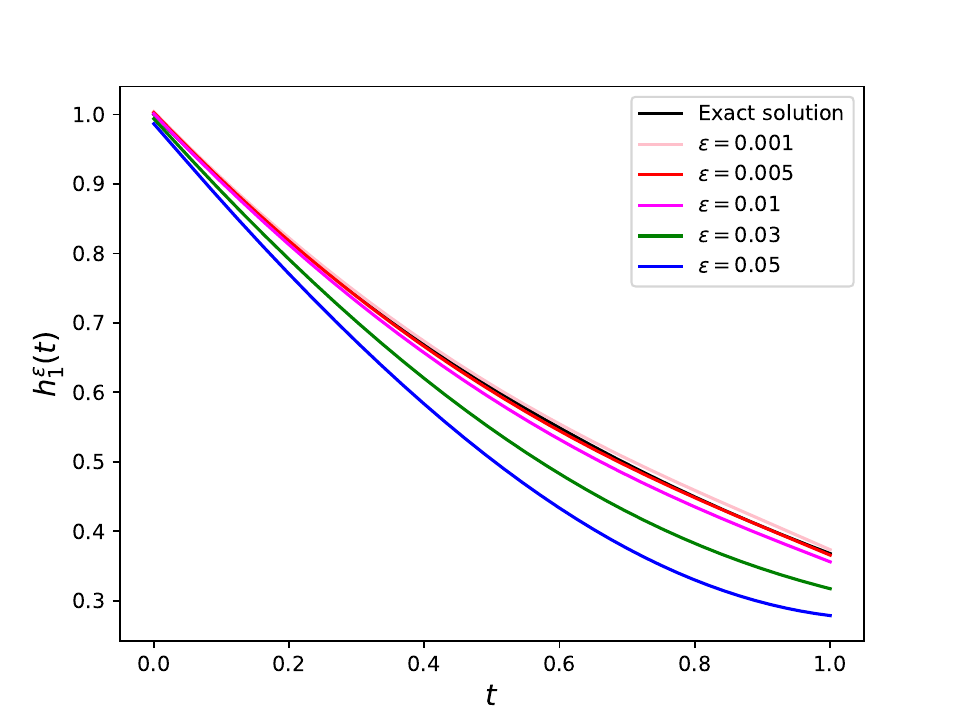}}
	\subfigure[]{\includegraphics[width=0.75\textwidth,angle=0,height = 0.22\textheight]{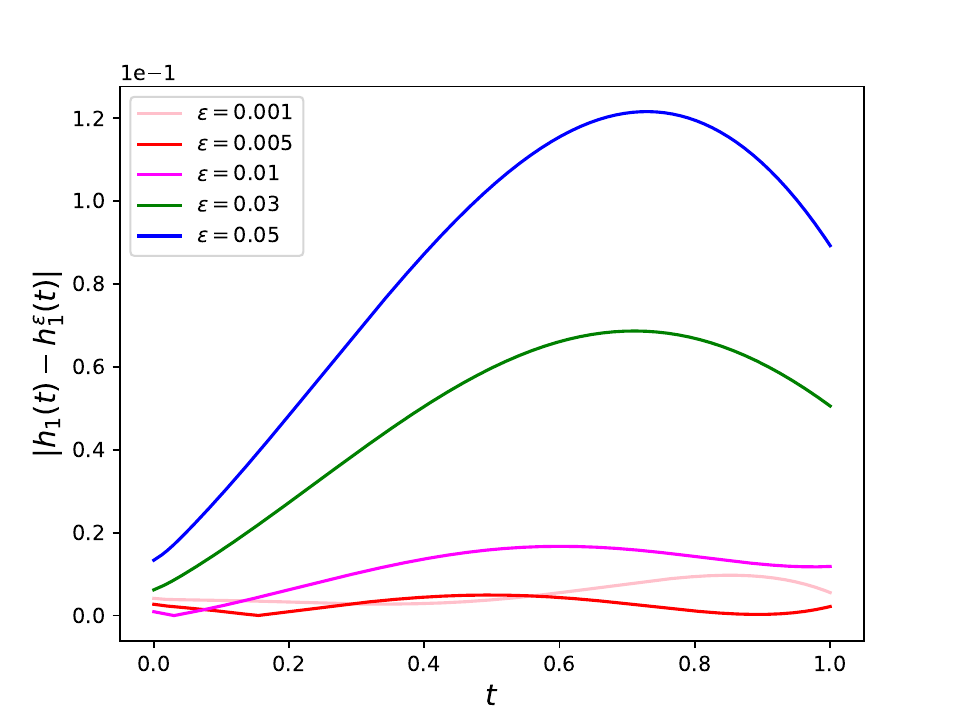}}
\end{center}
\vspace{-0.5cm}
\caption[Unknown time source (1D): Results for Experiment 1]{Experiment 1: (a)~The exact source and its numerical approximation, and (b)~its corresponding absolute error, obtained for various levels of noise.
}
\label{Experiment1}
\end{figure}
\begin{figure}[htbp]
\begin{center}
	\subfigure[]{\includegraphics[width=0.75\textwidth,angle=0,height = 0.22\textheight]{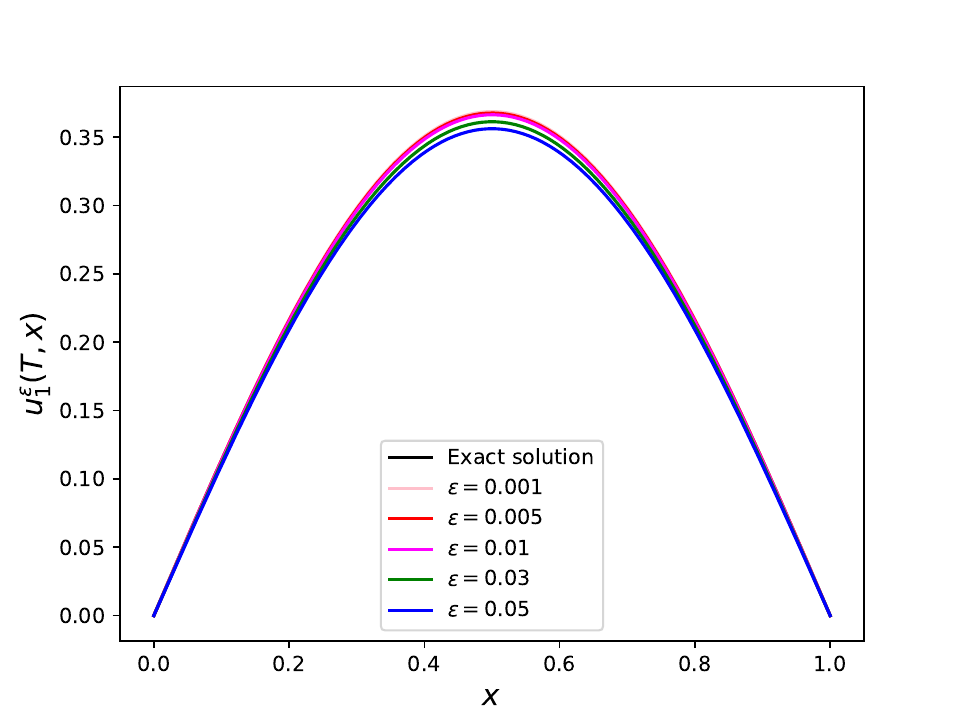}}
	\subfigure[]{\includegraphics[width=0.75\textwidth,angle=0,height = 0.22\textheight]{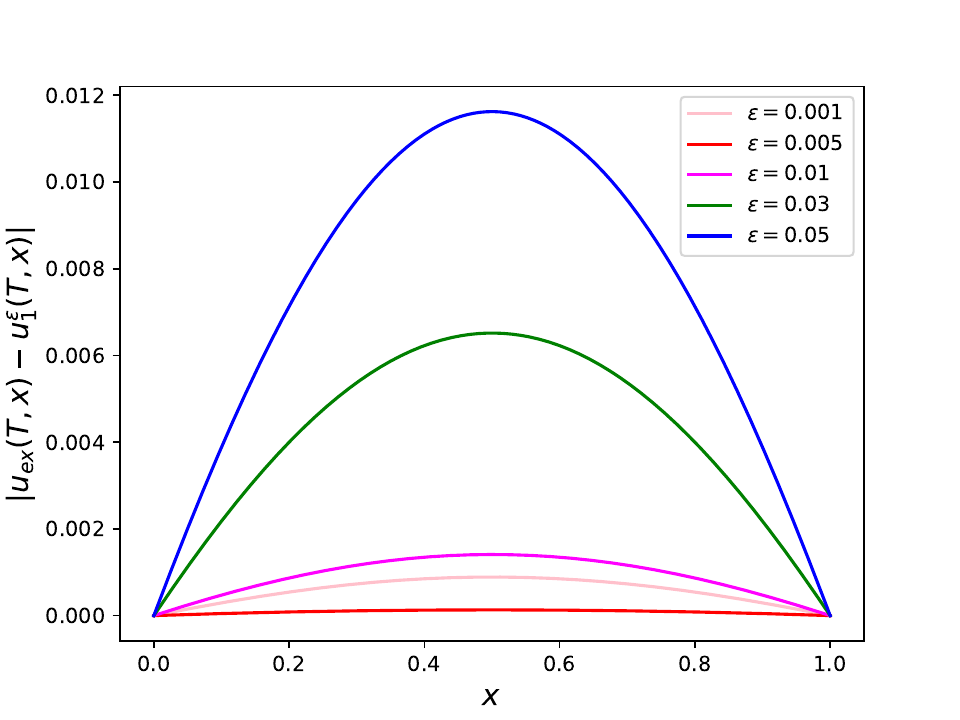}}
\end{center}
\vspace{-0.5cm}
\caption{Experiment 1 (1D): (a)~The exact solution at final time $T=1$ and its numerical approximation, and (b)~its corresponding absolute error, obtained for various levels of noise.}
\label{Experiment1u}
\end{figure}

\begin{figure}[htbp]
\begin{center}
	\subfigure[]{\includegraphics[width=0.75\textwidth,angle=0,height = 0.22\textheight]{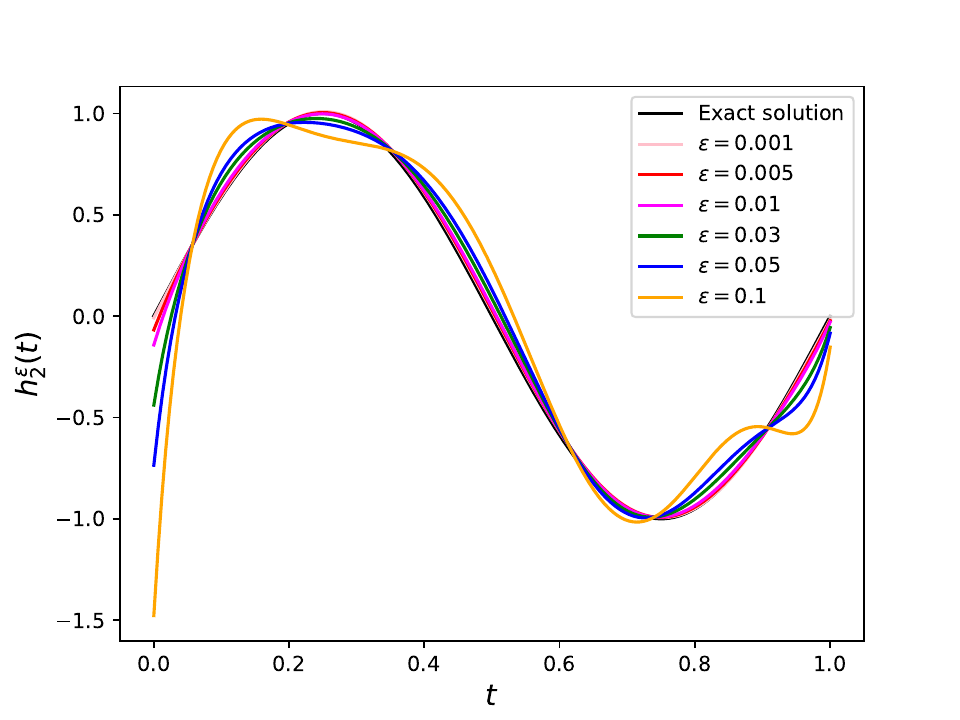}}
	\subfigure[]{\includegraphics[width=0.75\textwidth,angle=0,height = 0.22\textheight]{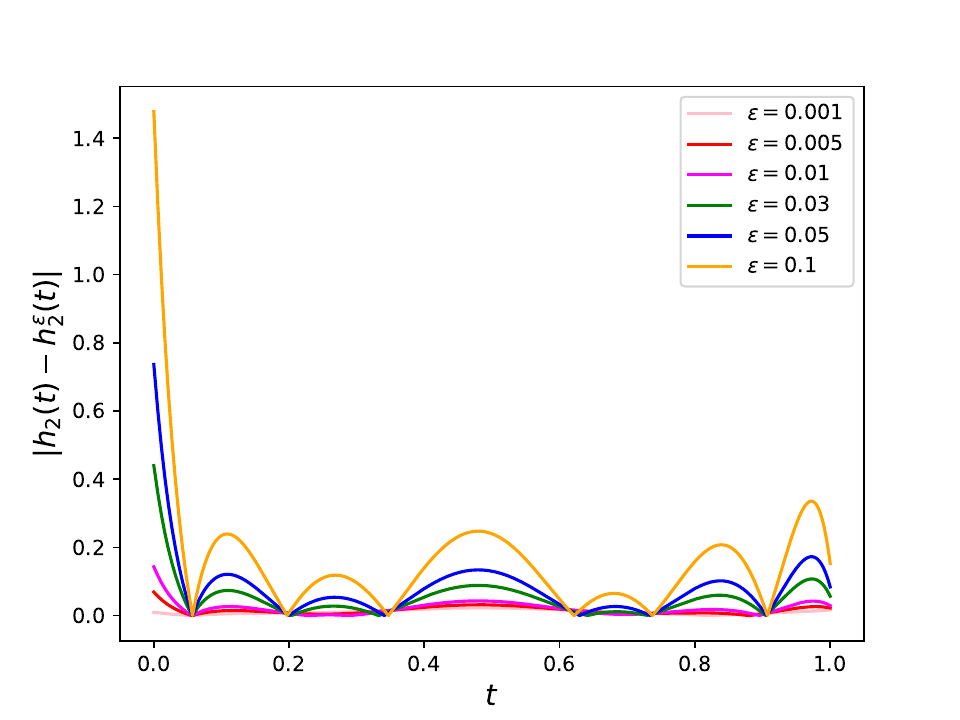}}
\end{center}
\vspace{-0.5cm}
\caption[Unknown time source: Results for Experiment 2]{Experiment 2 (1D): (a)~The exact source and its numerical approximation, and (b)~its corresponding absolute error, obtained for various levels of noise.
}
\label{Experiment2}
\end{figure}

\begin{figure}[htbp]
\begin{center}
	\subfigure[]{\includegraphics[width=0.75\textwidth,angle=0,height = 0.22\textheight]{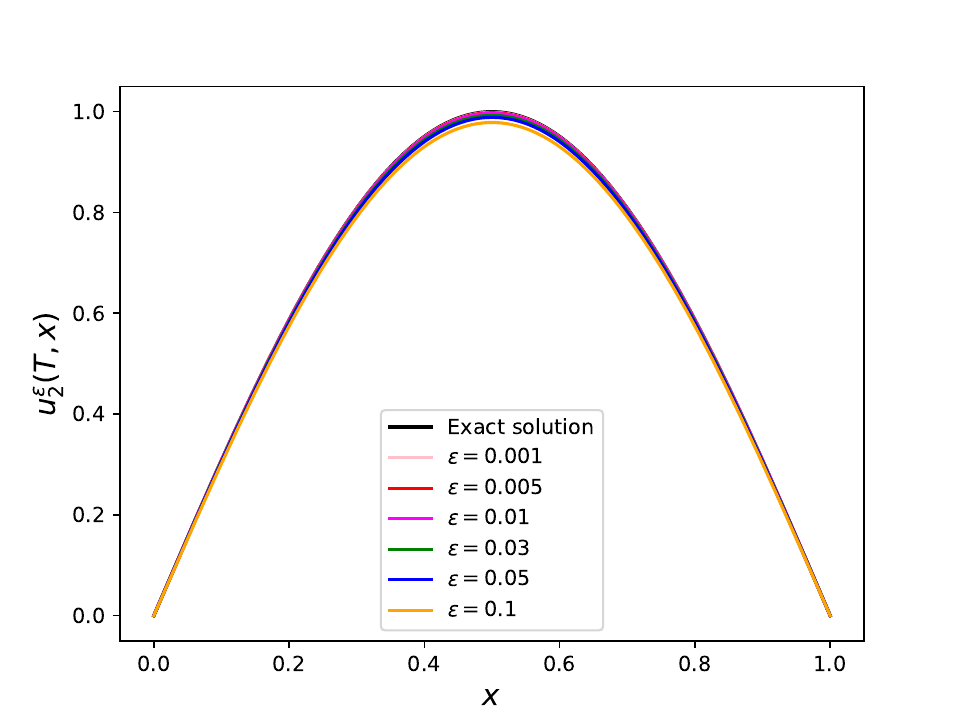}}
	\subfigure[]{\includegraphics[width=0.75\textwidth,angle=0,height = 0.22\textheight]{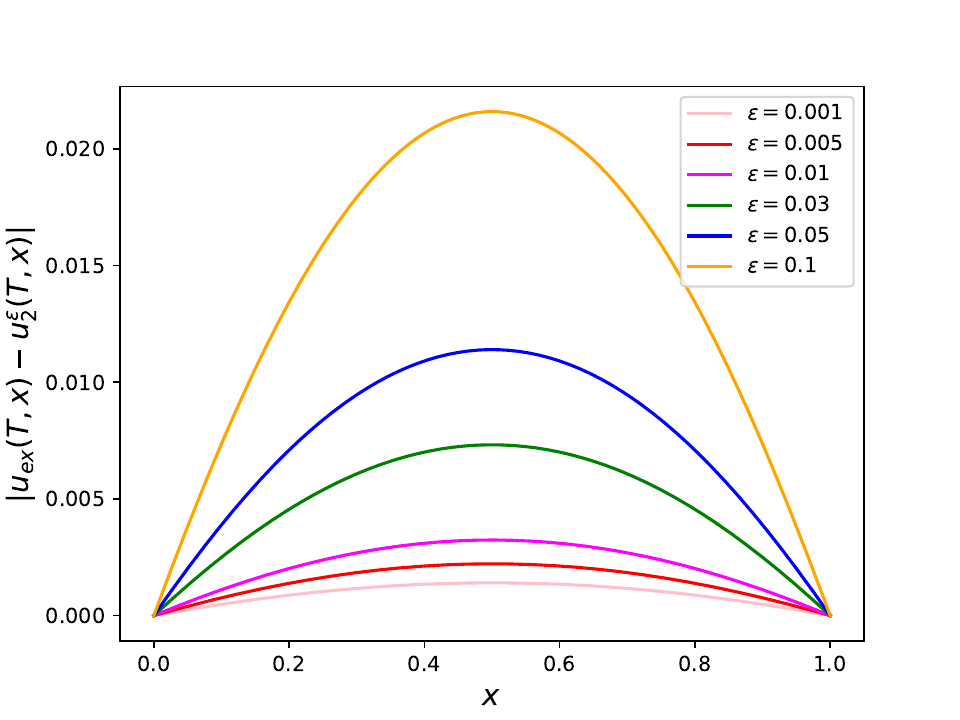}}
\end{center}
\vspace{-0.5cm}
\caption{Experiment 2 (1D): (a)~The exact solution at final time $T=1$ and its numerical approximation, and (b)~its corresponding absolute error, obtained for various levels of noise.}
\label{Experiment2u}
\end{figure}

\section{Conclusion}
\label{sec:conclusion}

In this paper, we investigated the inverse problem of recovering a time-dependent source term in a semilinear pseudo-parabolic equation with variable coefficients and Dirichlet boundary condition. The unknown source component $h(t)$ was determined from additional measurement data expressed as a weighted spatial average of the solution state.

The main theoretical contribution is the proof of existence and uniqueness of a weak solution to the inverse problem. This was achieved using Rothe's time-discretisation method, which required a non-standard treatment of the time-discrete problems due to the interplay between the third-order pseudo-parabolic term $\nabla \cdot (\eta \nabla \partial_t u)$ and the Dirichlet boundary condition. The analysis led to a source identification condition \eqref{cond:uniqueness} on the data, which guarantees the well-posedness of the inverse problem.

From a numerical perspective, we have implemented the inverse solver using a perturbation approach. At each time step, the solution of the discrete problem is obtained via the representation \eqref{eq:solution_representation}, which reduces the computation to solving two standard variational problems and evaluating a scalar coefficient. This approach is computationally efficient and robust for practical time discretisations. We have also established discrete-time error estimates in \Cref{thm:IP_disc_error}, showing that, under sufficient regularity assumptions on the data, the discrete solutions $(u_i,h_i)$ converge linearly in time to the exact solution $(u(t_i),h(t_i))$ as the timestep $\tau \to 0$. The numerical experiments confirm these theoretical rates and further demonstrate that the scheme yields accurate and stable reconstructions of both the solution and the source term, even for small noise levels.

For future work, it would be valuable to explore strategies for relaxing the source identification condition required in the theoretical analysis.

\bibliographystyle{unsrt} 
\bibliography{refs}

\end{document}